\documentclass[10pt,reqno]{amsart}
\usepackage[T1]{fontenc}
\usepackage[latin1]{inputenc}
\usepackage{color}
\usepackage{amsthm}
\usepackage{amstext}
\usepackage{amssymb}
\usepackage{esint}
\usepackage[all]{xy}
\usepackage[unicode=true,pdfusetitle,
 bookmarks=true,bookmarksnumbered=false,bookmarksopen=false,
 breaklinks=false,pdfborder={0 0 0},backref=false,colorlinks=true]
 {hyperref}

\makeatletter
\numberwithin{equation}{section}
\numberwithin{figure}{section}
\usepackage{enumitem}		
\theoremstyle{plain}
\newtheorem{thm}{\protect\theoremname}
  \theoremstyle{definition}
  \newtheorem{defn}[thm]{\protect\definitionname}
  \theoremstyle{definition}
  \newtheorem{example}[thm]{\protect\examplename}
  \theoremstyle{plain}
  \newtheorem{lem}[thm]{\protect\lemmaname}
  \theoremstyle{plain}
  \newtheorem{cor}[thm]{\protect\corollaryname}

\@ifundefined{date}{}{\date{}}
\usepackage{mathabx}
\usepackage{pdfsync}
\usepackage{graphics}
\usepackage{epstopdf}
\usepackage[all]{xy}
\usepackage{amscd}

\newcommand{\xyR}[1]{ \makeatletter
\xydef@\xymatrixrowsep@{#1} \makeatother} 
\newcommand{\xyC}[1]{ \makeatletter
\xydef@\xymatrixcolsep@{#1} \makeatother} 
\entrymodifiers={++[ ][F]} 

\newcommand{\ra}{\longrightarrow}

\newcommand{\field}[1]{\mathbb{#1}}
\newcommand{\R}{\field{R}} 
\newcommand{\N}{\field{N}} 

\newcommand{\D}{\mathcal{D}}

\newcommand{\eps}{\varepsilon} 
\renewcommand{\phi}{\varphi} 
\newcommand{\diff}[1]{\,\hbox{\rm d}#1} 

\newcommand{\Coo}{\mbox{\ensuremath{\mathcal{C}}}^{\infty}} 
\newcommand{\OR}{{\mathcal{O}}\mbox{$\R^\infty$}} 

\newcommand{\Rtil}{\widetilde \R} 
\newcommand{\gse}{{\mathcal G}^{\text{\rm e}}} 
\newcommand{\gsd}{{\mathcal G}^{\text{\rm d}}} 
\newcommand{\erm}{\text{\rm e}} 
\newcommand{\srm}{\text{\rm s}}
\newcommand{\drm}{\text{\rm d}}

\newcommand{\Colomb}{\mathcal{C}\!\text{\rm \emph{o}}\ell }

\newcommand{\maC}{\mathcal{C}}

\newcommand{\B}{\mathcal{B}}
\newcommand{\GBZ}{\mathcal{G}(\mathcal{B},\mathcal{Z},\Omega)}

\makeatother

  \providecommand{\corollaryname}{Corollary}
  \providecommand{\definitionname}{Definition}
  \providecommand{\examplename}{Example}
  \providecommand{\lemmaname}{Lemma}
\providecommand{\theoremname}{Theorem}

\begin{document}

\title{The category of Colombeau algebras}

\author{Lorenzo Luperi Baglini \and Paolo Giordano}

\thanks{P.~Giordano has been $\text{supp}$orted by grants P25116-N25 and
P25311-N25 of the Austrian Science Fund FWF}

\address{\textsc{Faculty of Mathematics, University of Vienna, Austria, Oskar-Morgenstern-Platz
1, 1090 Wien, Austria}}

\thanks{L.~Luperi Baglini has been supported by grants P25311-N25 and M1876-N35
of the Austrian Science Fund FWF}

\email{\texttt{paolo.giordano@univie.ac.at, lorenzo.luperi.baglini@univie.ac.at}}
\begin{abstract}
In \cite{GiLu15}, we introduced the notion of asymptotic gauge (AG),
and we used it to construct Colombeau AG-algebras. This construction
concurrently generalizes that of many different algebras used in Colombeau's
theory, e.g. the special one $\mathcal{G}^{\srm}$, the full one $\gse$,
the NSA based algebra of asymptotic functions $\hat{\mathcal{G}}$,
and the diffeomorphism invariant algebras $\gsd$, $\mathcal{G}^{2}$
and $\hat{\mathcal{G}}$. In this paper we study the categorical properties
of the construction of Colombeau AG-algebras with respect to the choice
of the AG, and we show their consequences regarding the solvability
of generalized ODE.
\end{abstract}

\maketitle

\section{Introduction}

Although Colombeau algebras were introduced to find solutions of differential
problems which are not solvable in classical spaces of distributions,
it is well known that very simple equations remains not solvable also
in these algebras, see e.g. \cite{GKOS,GiLu15}. A step toward the
analysis of these problems is the generalization of the role of the
infinite nets $(\eps^{-n})\in\R^{(0,1]}$ appearing in the definition
of Colombeau algebras. This has already been done through the notions
of asymptotic scale, $(\mathcal{C},\mathcal{E},\mathcal{P})$-algebras,
$(\mathcal{M},\mathcal{N},V_{\mathcal{P}})$-algebras, exponent weights
and asymptotic gauges (AG), see \cite{GiLu15} and references therein.
In particular, if one considers the usual sheaf of smooth functions,
Colombeau AG-algebras is the simplest and most general approach. In
fact, Colombeau AG-algebras include different algebras, like the special
one $\mathcal{G}^{\srm}$, the full one $\gse$, the NSA based algebra
of asymptotic functions $\hat{\mathcal{G}}$, and the diffeomorphism
invariant algebras $\gsd$, $\mathcal{G}^{2}$ and $\hat{\mathcal{G}}$
(see \cite{GKOS}). Its simplicity lies, for all these algebras, in
the use of the simple logical structure of quantifiers that characterizes
the special algebra $\mathcal{G}^{\srm}$.

In the context of AG, it is therefore natural to set the following
questions:
\begin{itemize}
\item Is the construction of the Colombeau algebra functorial with respect
to the AG? Is this construction functorial with respect to the open
set $\Omega$?
\item When can we consider two AG as isomorphic? For instance, we will show
that the AG of polynomial growth is isomorphic to the AG of exponential
growth. This isomorphism holds in spite of the fact that using the
latter we can solve ODE which are not solvable with the former, see
Sec.~\ref{sec:An-unexpected-isomorphism}.
\item How to relate the solutions obtained using one AG to those obtained
using another one?
\item Colombeau theory can be more clearly summarized by saying that it
permits to define a differential algebra together with an embedding
of Sch\-war\-tz's distributions. This embedding can be intrinsic,
or diffeomorphism invariant, or it can be chosen in order to have
properties like $H(0)=\frac{1}{2}$, where $H$ is the Heaviside's
step function. Can we define a general category having as objects
triples $(G,\partial,i)$ made of an algebra $G$, a family of derivations
$\partial$ and an embedding of distributions $i$? Can we see $\mathcal{G}^{\srm}$
as a suitable functor with values in this category? What is the domain
of this functor?
\end{itemize}
\noindent In the present work, we answer these questions.

\section{Sets of Indices}

\subsection{Basic definitions}

In \cite{GiNi14}, the general notion of sets of indices has been
introduced. This notion permits to unify the presentation of several
Colombeau-type algebras of nonlinear generalized functions. For reader's
convenience, in this section we recall the notations and notions from
\cite{GiNi14} that we will use in the present work. For all the proofs,
we refer to \cite{GiNi14}.
\begin{defn}
\label{def:setOfIndices}We say that $\mathbb{I}=(I,\le,\mathcal{I})$
is a \emph{set of indices} if the following conditions hold: 
\begin{enumerate}[leftmargin=*,label=(\roman*),align=left ]
\item  $(I,\le)$ is a pre-ordered set, i.e., $I$ is a non empty set with
a reflexive and transitive relation $\le$ ;
\item $\mathcal{I}$ is a set of subsets of $I$ such that $\emptyset\notin\mathcal{I}$
and $I\in\mathcal{I}$;
\item \label{enu:SoI-Filter}$\forall A,B\in\mathcal{I}\,\exists C\in\mathcal{I}:\ C\subseteq A\cap B$.
\end{enumerate}

For all $e\in I$, set $(\emptyset,e]:=\left\{ \eps\in I\mid\eps\le e\right\} $.
As usual, we say $\eps<e$ if $\eps\le e$ and $\eps\ne e$. Using
these notations, we state the last condition in the definition of
set of indices:
\begin{enumerate}[leftmargin=*,label=(\roman*),align=left,start=4]
\item \label{enu:SoI-directed}If $e\le a\in A\in\mathcal{I}$ , the set
$A_{\le e}:=(\emptyset,e]\cap A$ is downward directed by $<$ , i.e.,
it is non empty and $\forall b,c\in A_{\le e}\,\exists d\in A_{\le e}:\ d<b,\ d<c$.
\end{enumerate}
\end{defn}
\noindent The following are examples of sets of indices.
\begin{example}
\label{exa:setsIndices}\ 
\begin{enumerate}[leftmargin=*,label=(\roman*),align=left]
\item \label{enu:I^s} Let $I^{\srm}:=(0,1]\subseteq\R$, let $\le$ be
the usual order relation on $\R$, and let $\mathcal{I}^{\srm}:=\left\{ (0,\eps_{0}]\mid\eps_{0}\in I\right\} $.
Following \cite{GiNi14}, we denote by $\mathbb{I}^{\srm}:=(I^{\srm},\le,\mathcal{I}^{\srm})$
this set of indices.
\item \label{enu:I^e}If $\phi\in\D(\R^{n}),\, r\in\R_{>0}$ and $x\in\R^{n}$,
we use the symbol $r\odot\phi$ to denote the function $x\in\R^{n}\mapsto\frac{1}{r^{n}}\cdot\phi\left(\frac{x}{r}\right)\in\R$,
see \cite{GiNi14}. With the usual notations of \cite{GKOS}, we define
$I^{\erm}:=\mathcal{A}_{0},\mathcal{\, I}^{\erm}:=\left\{ \mathcal{A}_{q}\mid q\in\N\right\} $,
and for $\eps,\, e\in I^{\erm}$, we set $\eps\le e$ iff there exists
$r\in\R_{>0}$ such that $r\le1$ and $\eps=r\odot e$. Then $\mathbb{I}^{\erm}:=(I^{\erm},\le,\mathcal{I}^{\erm})$
is a set of indices used in this framework to unify and simplify the
full algebra $\gse$ (see \cite[Sec.~3]{GiNi14}).
\item \label{enu:I-tilde-e}For every $\varphi\in\mathcal{A}_{0}$, let
us call \emph{order of }$\varphi$ the natural number
\[
o(\varphi):=\min{\{n\in\mathbb{N}\mid\varphi\in\mathcal{A}_{n}\setminus\mathcal{A}_{n+1}\}}
\]
and, for every $\varphi$, $\psi\in\mathcal{A}_{0}$, set 
\[
\varphi\lesssim\psi\mbox{ iff }o(\varphi)<o(\psi)\,\mbox{or\,}\varphi\leq\psi\,\mbox{in}\,\mathbb{I^{\erm}}.
\]
We have that $\widetilde{\mathbb{I}}^{\erm}=(\mathcal{\mathcal{A}}_{0},\lesssim,\{\mathcal{A}_{q}\mid q\in\mathbb{N}\})$
is a downward directed set of indices that can be used to try a simplification
of the full algebra $\gse$. See Sec.~\ref{sub:Downward-directed-sets}
for the nicer properties that downward directed sets have also with
respect to the notions we are going to introduce. 
\end{enumerate}
\end{example}
Henceforward, functions of the type $f:I\ra\R$ will also be called
\emph{nets}, and for their evaluation we will both use the notations
$f_{\eps}$ or $f(\eps)$, the latter in case the subscript notation
is too cumbersome. When the domain $I$ is clear, we use also the
notation $f=\left(f_{\eps}\right)$ for the whole net. Analogous notations
will be used for nets of smooth functions $u=(u_{\eps})\in\mathcal{C}^{\infty}(\Omega)^{I}$.

In each set of indices, we can define two notions of big-O for nets
of real numbers. These two big-Os share the same (usual) properties
of the classical one as preorders and concerning algebraic operations
(see \cite[Thm.~2.8, Thm.~2.14]{GiNi14}). Since each set of the form
$A_{\le e}=(\emptyset,e]\cap A$ is downward directed, the first big-O
is the usual one:
\begin{defn}
\label{def:usualBigOh}Let $\mathbb{I}=(I,\le,\mathcal{I})$ be a
set of indices. Let $a\in A\in\mathcal{I}$ and let $(x_{\eps})$,
$(y_{\eps})\in\R^{I}$ be two nets of real numbers defined in $I$.
We write
\begin{equation}
x_{\eps}=O_{a,A}(y_{\eps})\ \text{as }\eps\in\mathbb{I}\label{eq:usualBigOh}
\end{equation}
if
\begin{equation}
\exists H\in\R_{>0}\,\exists\eps_{0}\in A_{\le a}\,\forall\eps\in A_{\le\eps_{0}}:\ |x_{\eps}|\le H\cdot|y_{\eps}|.\label{eq:1stBigOhDef}
\end{equation}

\end{defn}
\noindent The second notion of big-O is the following:
\begin{defn}
\label{def:2ndBigOh}Let $\mathbb{I}=(I,\le\mathcal{I})$ be a set
of indices. Let $\mathcal{J}\subseteq\mathcal{I}$ be a non empty
subset of $\mathcal{I}$ such that
\begin{equation}
\forall A,B\in\mathcal{J}\,\exists C\in\mathcal{J}:\ C\subseteq A\cap B.\label{eq:hypJ}
\end{equation}
Finally, let $(x_{\eps})$, $(y_{\eps})\in\R^{I}$ be nets of real
numbers. Then we say
\[
x_{\eps}=O_{\mathcal{J}}(y_{\eps})\text{ as }\eps\in\mathbb{I}
\]
if
\[
\exists A\in\mathcal{J}\,\forall a\in A:\ x_{\eps}=O_{a,A}(y_{\eps}).
\]
We simply write $x_{\eps}=O(y_{\eps})$ (as $\eps\in\mathbb{I}$)
when $\mathcal{J=I}$, i.e.~to denote $x_{\eps}=O_{\mathcal{I}}(y_{\eps})$.
\end{defn}
For example, in case of the set of indices $\mathbb{I}^{\erm}$ used
for the full algebra, we have $x_{\eps}=O(y_{\eps})$ as $\eps\in\mathbb{I}^{\erm}$
if and only if $\exists q\in\N\,\forall\phi\in\mathcal{A}_{q}:\ x(\eps\odot\phi)=O\left[y(\eps\odot\phi)\right]$
as $\eps\to0^{+}$, where the latter big-O is the classical one, see
\cite[Thm. ~3.2]{GiNi14}. We can hence recognise an important part
of the usual definition of moderate and negligible nets for the full
algebra $\gse$. The abstract approach we use in this paper can be
easily understood by interpreting $\mathbb{I}$ in the simplest case
$\mathbb{I}^{\srm}$ of the special algebra and in the case $\mathbb{I}^{\erm}$
of the full algebra. In the former, any formula of the form $\exists A\in\mathcal{I}\,\forall a\in A$
becomes $\exists\eps_{0}\in(0,1]\,\forall\eps\in(0,\eps_{0}]$. In
the latter it becomes $\exists q\in\N\,\forall\phi\in\mathcal{A}_{q}$.

In every set of indices we can formalize the notion of \emph{for $\eps$
sufficiently small} as follows.
\begin{defn}
\label{def:epsSuffSmall}Let $\mathbb{I}=(I,\le,\mathcal{I})$ be
a set of indices. Let $a\in A\in\mathcal{I}$ and $\mathcal{P}(-)$
be a property, then we say
\[
\forall^{\mathbb{I}}\eps\in A_{\le a}:\ \mathcal{P}(\eps),
\]
and we read it \emph{for $\eps$ sufficiently small in $A_{\le a}$
the property $\mathcal{P}(\eps)$ holds}, if 
\begin{equation}
\exists e\in A_{\le a}\,\forall\eps\in A_{\le e}:\ \mathcal{P}(\eps).\label{eq:defSuffSmall}
\end{equation}
Note that, by condition \ref{enu:SoI-directed} of Def.~\ref{def:setOfIndices},
it follows that $A_{\le e}\ne\emptyset$, so that \eqref{eq:defSuffSmall}
is equivalent to
\[
\exists e\le a\,\forall\eps\in A_{\le e}:\ \mathcal{P}(\eps).
\]
Moreover, we say that
\[
\forall^{\mathbb{I}}\eps:\ \mathcal{P}(\eps),
\]
and we read it \emph{for $\eps$ sufficiently small in $\mathbb{I}$
the property $\mathcal{P}(\eps)$ holds, if $\exists A\in\mathcal{I}\,\forall a\in A\,\forall^{\mathbb{I}}\eps\in A_{\le a}:\ \mathcal{P}(\eps)$.}
\end{defn}
\noindent Using this notion, we can define an order relation for nets.
\begin{defn}
\label{def:order}Let $\mathbb{I}=(I,\le,\mathcal{I})$ be a set of
indices, and $i$, $j:I\ra\R$ be nets. Then we say $i>_{\mathbb{I}}j$
if
\[
\forall^{\mathbb{I}}\eps:\ i_{\eps}>j_{\eps}.
\]

\end{defn}
\noindent Finally, we recall the notion of limit of a net of real
numbers:
\begin{defn}
\label{def:limit}Let $\mathbb{I}=(I,\le,\mathcal{I})$ be a set of
indices, $f:I\ra\R$ a map, and $l\in\R\cup\{+\infty,-\infty\}$.
Then we say that \emph{$l$ is the limit of $f$ in $\mathbb{I}$}
if
\begin{equation}
\exists A\in\mathcal{I}\,\forall a\in A:\ l=\lim_{\eps\le a}f|_{A}(\eps),\label{eq:limitDef}
\end{equation}
where the limit \eqref{eq:limitDef} is taken in the downward directed
set $(\emptyset,a]=I_{\le a}$.
\end{defn}
Let us observe that if $l=\lim_{\eps\le a}f|_{A}(\eps)$ and $B\subseteq A$
, $B\in\mathcal{I}$ , then $l=\lim_{\eps\le a}f|_{B}(\eps)$; moreover,
there exists at most one $l$ verifying \eqref{eq:limitDef}.

\section{The Category $\textbf{Ind}$}

We start by defining the notion of morphism between two sets of indices.
This is also a natural step to define the concept of morphism of asymptotic
gauges. A natural property to expect from a morphism $f:\mathbb{I}_{1}\ra\mathbb{I}_{2}$
between sets of indices $\mathbb{I}_{1}$, $\mathbb{I}_{2}$ is the
preservation of the notion of ``eventually'' for properties $\mathcal{P}$,
i.e.~that $\forall^{\mathbb{I}_{1}}\varepsilon_{1}\,\mathcal{P}(\varepsilon_{1})$
implies $\forall^{\mathbb{I}_{2}}\varepsilon_{2}\,\mathcal{P}(f(\varepsilon_{2}))$.
Let us note that we start from a property $\mathcal{P}(\eps_{1})$,
for $\eps_{1}\in I_{1}$, and we want to arrive at a property $\mathcal{P}(f(\eps_{2}))$,
for $\eps_{2}\in I_{2}$.
\begin{defn}
\label{def:MorphInd}Let $\mathbb{I}_{k}=(I_{k},\leq_{k},\mathcal{I}_{k})$
be sets of indices for $k=1$, $2$. Let $a\in A\in\mathcal{I}_{1},\, b\in B\in\mathcal{I}_{2}$.
Then we say that $f:A_{\leq a}\longrightarrow B_{\leq b}$ is \emph{infinitesimal}
if 
\begin{enumerate}[leftmargin=*,label=(\roman*),align=left]
\item $f:I_{2}\longrightarrow I_{1}$;
\item $\forall\alpha\in A_{\leq a}\,\forall^{\mathbb{I}_{2}}\varepsilon_{2}\in B_{\leq b}:\ f(\varepsilon_{2})\in A_{\leq\alpha}$. 
\end{enumerate}

Moreover, we say that $f:\mathbb{I}_{1}\longrightarrow\mathbb{I}_{2}$
is a \textit{morphism} of sets of indices if
\[
\forall A\in\mathcal{I}_{1}\,\forall a\in A\,\exists B\in\mathcal{I}_{2}\,\forall b\in B:\ f:A_{\leq a}\longrightarrow B_{\leq b}\text{ is infinitesimal.}
\]

\end{defn}
\noindent Therefore, a morphism $f:\mathbb{I}_{1}\longrightarrow\mathbb{I}_{2}$
is a map in the opposite direction $f:I_{2}\longrightarrow I_{1}$
between the underlying sets. Only in this way we have that the map
$f$ preserves the asymptotic relations that hold in $\mathbb{I}_{1}$,
see Cor.~\ref{cor:morfind} for a list of examples.
\begin{example}
\label{exa:morphSetsOfIndices}$\,$
\begin{enumerate}[leftmargin=*,label=(\roman*),align=left]
\item For every set of indices $\mathbb{I}=(I,\leq,\mathcal{I})$ if $1_{\mathbb{I}}:I\longrightarrow I$
is the identity function then $1_{\mathbb{I}}:\mathbb{I\longrightarrow\mathbb{I}}$
is a morphism.
\item \label{enu:morph-I^s}Let $f:(0,1]\ra(0,1]$ be a map, then $f:\mathbb{I}^{\srm}\ra\mathbb{I}^{\srm}$
is a morphism if and only if $\forall\eps\in(0,1]\,\exists\delta\in(0,1]:\ f\left((0,\delta]\right)\subseteq(0,\eps]$,
i.e.~if and only if $\lim_{\eps\to0^{+}}f(\eps)=0$.
\item For the set of indices $\mathbb{I}^{\erm}$ of the full algebra, we
recall that $\left(\mathcal{A}_{q}\right)_{\le\phi}=(\emptyset,\phi]$
and $\underline{\phi}:=\min\left\{ \text{diam}(\text{supp}\phi),1\right\} $.
If $f:I^{\erm}\ra I^{\erm}$ is a map, then we have that $f:(\emptyset,\phi]\ra(\emptyset,\psi]$
is infinitesimal if and only if $\forall\eps\in(0,1]\,\exists\delta\in(0,1]:\ f\left(\left\{ r\odot\psi\mid r\in(0,\delta]\right\} \right)\subseteq\left\{ r\odot\phi\mid r\in(0,\eps]\right\} $.
Therefore, this implies that $\lim_{\eps\to0^{+}}\underline{f(\eps\odot\psi)}=0$.
If we denote by $\frac{f(\eps\odot\psi)}{\phi}$ the unique $r\in(0,1]$
such that $f(\eps\odot\psi)=r\odot\phi$ (in case it exists), then
$f:(\emptyset,\phi]\ra(\emptyset,\psi]$ is infinitesimal if and only
if $\lim_{\eps\to0^{+}}\frac{f(\eps\odot\psi)}{\phi}=0$. Moreover,
$f:\mathbb{I}^{\erm}\ra\mathbb{I}^{\erm}$ is a morphism if and only
if $\forall m\in\N\,\forall\phi\in\mathcal{A}_{m}\,\exists q\in\N\,\forall\psi\in\mathcal{A}_{q}:\ \lim_{\eps\to0^{+}}\frac{f(\eps\odot\psi)}{\phi}=0$.
This and the previous example justify our use of the name \emph{infinitesimal}
in Def.~\ref{def:MorphInd}.
\item \label{enu:phiFixed}Let $\varphi\in\mathcal{A}_{0}$ be fixed, let
$\mathbb{I}_{\varphi}:=((\emptyset,\varphi],\leq,\{(\emptyset,\varphi]\})$,
where the order relation on $\mathbb{I_{\varphi}}$ is the restriction
of the order relation on $\mathbb{I}^{\erm}$. If $f:(0,1]\longrightarrow(\emptyset,\varphi]$
is the function $f(r):=r\odot\varphi$ for every $r\in(0,1]$ then
we have that $f:\mathbb{I}_{\varphi}\longrightarrow\mathbb{I^{\srm}}$
is a morphism. Conversely, if $g:(\emptyset,\varphi]\longrightarrow(0,1]$
maps every $\psi\in(\emptyset,\varphi]$ to the unique $r\in(0,1]$
such that $\psi=r\odot\varphi$, i.e.~$g(\psi)=\frac{\psi}{\phi}$,
then $g:\mathbb{I}^{\srm}\longrightarrow\mathbb{I}_{\varphi}$ is
a morphism. We have that $f=g^{-1}$.
\item \label{enu:Nbar}Let us denote by $\overline{\mathbb{N}}$ the set
of indices $(\mathbb{N},\lesssim,\mathcal{I}_{n})$ where $\lesssim$
is the inverse of the usual order notion on $\mathbb{N}$ (namely,
$m\lesssim n$ iff $m\ge n$) and, for every natural number $n$,
$\mathcal{I}_{n}:=\{m\in\mathbb{N}\mid m\lesssim n\}.$ If $f:\mathbb{N}\rightarrow(0,1]$
is the function that maps $n$ to $\frac{1}{n+1}$, we have that $f:\mathbb{I}^{\srm}\longrightarrow\overline{\mathbb{N}}$
is a morphism. Conversely, if $g:(0,1]\longrightarrow\mathbb{N}$
is the function that maps $\varepsilon$ to the floor $\left\lfloor \frac{1}{\varepsilon}\right\rfloor $
then $g:\overline{\mathbb{N}}\longrightarrow\mathbb{I}^{\srm}$ is
a morphism. 
\item \label{enu:n-to-phi_n}For every $n\in\mathbb{N}$ let us fix $\varphi_{n}\in\mathcal{A}_{n}\setminus\mathcal{A}_{n+1}$.
Let $f:\mathbb{N}\longrightarrow\mathcal{A}_{0}$ be the function
that maps $n$ to $\varphi_{n}$. Then we have that $f:\widetilde{\mathbb{I}}^{\erm}\longrightarrow\overline{\mathbb{N}}$
is a morphism. Conversely, if $o:\mathcal{A}_{0}\longrightarrow\overline{\mathbb{N}}$
is the function that maps $\varphi$ to $o(\varphi)$ (see \ref{enu:I-tilde-e}
in Example~\ref{exa:setsIndices}) then $o:\overline{\mathbb{N}}\longrightarrow\widetilde{\mathbb{I}}^{\erm}$
is a morphism. 
\end{enumerate}
\end{example}
\begin{lem}
\label{lem:1stComposition}Let $\mathbb{I}_{k}=(I_{k},\leq_{k},\mathcal{I}_{k})$
be sets of indices for $k=1$, $2$, $3$. Let $a\in A\in\mathcal{I}_{1},\, b\in B\in\mathcal{I}_{2}$
and $c\in C\in\mathcal{I}_{3}$. Then if $f:A_{\le a}\ra B_{\le b}$
and $g:B_{\le b}\ra C_{\le c}$ are infinitesimals, also the composition
$f\circ g:A_{\le a}\ra C_{\le c}$ is infinitesimal.\end{lem}
\begin{proof}
By definition, for every $\alpha\in A_{\le a}$ there exists $\delta_{2}\leq b$
such that $f(\eps_{2})\in A_{\le\alpha}$ for every $\eps_{2}\in B$
such that $\eps_{2}\le\delta_{2}\le b$. But $g:B_{\le b}\ra C_{\le c}$
infinitesimal means
\[
\forall\beta\in B_{\le b}\,\forall^{\mathbb{I}_{3}}\eps_{3}\in C_{\le c}:\ g(\eps_{3})\in B_{\le\beta}.
\]
We apply this property with $\beta=\delta_{2}$ to get $g(\eps_{3})\in B_{\le\delta_{2}}$
for every $\eps_{3}\in C_{\le c}$ sufficiently small, let us say
for each $\eps_{3}\le\delta_{3}\le c$. Therefore $f(g(\eps_{3}))\in A_{\le\alpha}$
for every $\eps_{3}\in C_{\le c}$ such that $\eps_{3}\le\delta_{3}$.
\end{proof}
The following results motivate our definition of morphism of sets
of indices.
\begin{lem}
\label{lem:MorphProp}In the assumptions of Def.~\ref{def:MorphInd},
let $f:A_{\leq a}\longrightarrow B_{\leq b}$ be infinitesimal, and
let $\mathcal{P}(\varepsilon_{1})$ be a given property of $\eps_{1}\in I_{1}$.
If $\forall^{\mathbb{I}_{1}}\varepsilon_{1}\in A_{\leq a}\,\mathcal{P}(\varepsilon_{1})$
then $\forall^{\mathbb{I}_{2}}\varepsilon_{2}\in B_{\leq b}\,\mathcal{P}(f(\varepsilon_{2}))$.\end{lem}
\begin{proof}
Let $e_{1}\in A_{\leq a}$ be such that $\mathcal{P}(\varepsilon_{1})$
holds for all $\varepsilon_{1}\in A_{\leq e_{1}}$. Since $f:A_{\leq a}\longrightarrow B_{\leq b}$
is infinitesimal, there exists $e_{2}\in B_{\leq b}$ be such that
$f(\varepsilon_{2})\in A_{\leq e_{1}}$ for all $\varepsilon_{2}\in B_{\leq e_{2}}$.
Therefore $\mathcal{P}(f(\varepsilon_{2}))$ holds for all $\varepsilon_{2}\in B_{\leq e_{2}}$.\end{proof}
\begin{thm}
\label{thm:MorphProp}Let $\mathbb{I}_{k}=(I_{k},\leq_{k},\mathcal{I}_{k})$
be sets of indices for $k=1,2$. Let $f:\mathbb{I}_{1}\longrightarrow\mathbb{I}_{2}$
be a morphism of sets of indices and let $\mathcal{P}(\varepsilon_{1})$
be a given property of $\eps_{1}\in I_{1}$. If $\forall^{\mathbb{I}_{1}}\varepsilon_{1}\,\mathcal{P}(\varepsilon_{1})$
then $\forall^{\mathbb{I}_{2}}\varepsilon_{2}\,\mathcal{P}(f(\varepsilon_{2}))$.\end{thm}
\begin{proof}
Let $A\in\mathcal{I}_{1}$ be such that $\forall a\in A\,\forall^{\mathbb{I}_{1}}\varepsilon_{1}\in A_{\leq a}\,\mathcal{P}(\varepsilon_{1})$
holds. Since $\emptyset\notin\mathcal{I}_{1}$, there exists $a\in A$.
But $f:\mathbb{I}_{1}\longrightarrow\mathbb{I}_{2}$ is a morphism,
so there exists $B\in\mathcal{I}_{2}$ such that $f:A_{\leq a}\longrightarrow B_{\leq b}$
is infinitesimal for all $b\in B$. By Lemma~\ref{lem:MorphProp},
we deduce that $\forall^{\mathbb{I}_{2}}\varepsilon_{2}\in B_{\leq b}\,\mathcal{P}(f(\varepsilon_{2}))$,
which is our conclusion.
\end{proof}
Three simple consequences of Theorem \ref{thm:MorphProp} are presented
in the following Corollary.
\begin{cor}
\label{cor:morfind}Let $\mathbb{I}_{k}=(I_{k},\leq_{k},\mathcal{I}_{k})$
be sets of indices for $k=1,2$. If $f:\,\mathbb{I}_{1}\longrightarrow\mathbb{I}_{2}$
is a morphism of sets of indices, then the following properties hold:
\begin{enumerate}[leftmargin=*,label=(\roman*),align=left]
\item \label{enu:morphismsPreserveOrder}If $i>_{\mathbb{I}_{1}}j$ then
$i\circ f>_{\mathbb{I}_{2}}j\circ f$;
\item \label{enu:morphismsPreserveBigO}If $x_{\varepsilon_{1}}=O(y_{\varepsilon_{1}})$
as $\eps_{1}\in\mathbb{I}_{1}$, then $x_{f(\varepsilon_{2})}=O(y_{f(\varepsilon_{2})})$
as $\eps_{2}\in\mathbb{I}_{2}$;
\item \label{enu:morphismsPreserveLim}For every net $g:I_{1}\longrightarrow\mathbb{R}$
if $l=\lim_{\mathbb{I}_{1}}g$ then $l=\lim_{\mathbb{I}_{2}}g\circ f$.
\end{enumerate}
\end{cor}
\begin{proof}
Property \ref{enu:morphismsPreserveOrder} follows directly from Thm.~\ref{thm:MorphProp}
because $i>_{\mathbb{I}_{1}}j$ means $\forall^{\mathbb{I}}\eps:\ i_{\eps}>j_{\eps}$.
To prove \ref{enu:morphismsPreserveBigO}, let $A\in\mathcal{I}_{1}$
be such that $x_{\eps_{1}}=O_{A,a}(y_{\eps_{1}})$ for all $a\in A$.
Therefore, there exists $H\in\R_{>0}$ such that $\forall^{\mathbb{I}_{1}}\eps_{1}\in A_{\le a}\ \left|x_{\eps_{1}}\right|\le H\cdot\left|y_{\eps_{1}}\right|$.
But $A\neq\emptyset$, so we can pick $a\in A$, and $f:\,\mathbb{I}_{1}\longrightarrow\mathbb{I}_{2}$
yields the existence of $B\in\mathcal{I}_{2}$ such that $f:A_{\le a}\ra B_{\le b}$
is infinitesimal for all $b\in B$. By Lemma~\ref{lem:MorphProp}
we get $\forall^{\mathbb{I}_{2}}\eps_{2}\in B_{\le b}\ \left|x_{f\left(\eps_{2}\right)}\right|\le H\cdot\left|y_{f\left(\eps_{2}\right)}\right|$,
from which the conclusion follows. Using the same ideas, we can prove
\ref{enu:morphismsPreserveLim}.\end{proof}
\begin{thm}
\label{thm:IndCategory}The class of all sets of indices together
with their morphisms form a category $\mathbf{Ind}$.\end{thm}
\begin{proof}
The only non-trivial property to prove concerns composition, namely
that for every pair of arrows $\mathbb{I}_{1}\overset{f}{\longrightarrow}\mathbb{I}_{2},\,\mathbb{I}_{2}\overset{g}{\longrightarrow}\mathbb{I}_{3}$,
we have that $\mathbb{I}_{1}\overset{f\circ g}{\longrightarrow}\mathbb{I}_{3}$
is a morphism of set of indices. By our hypotheses we know that:
\begin{align}
\forall A & \in\mathcal{I}_{1}\,\forall a\in A\,\exists B\in\mathcal{I}_{2}\,\forall b\in B:\ f:A_{\leq a}\ra B_{\leq b}\text{ is infinitesimal};\label{eq:Orribile1}\\
\forall B & \in\mathcal{I}_{2}\,\forall b\in B\,\exists C\in\mathcal{I}_{3}\,\forall c\in C:\ g:B_{\leq b}\ra C_{\leq c}\text{ is infinitesimal}.\label{eq:Orribile2}
\end{align}

\noindent For $a\in A\in\mathcal{I}_{1}$, from \eqref{eq:Orribile1}
we get a non empty $B\in\mathcal{I}_{2}$. Take any element $b\in B$,
so that \eqref{eq:Orribile2} yields the existence of $C\in\mathcal{I}_{3}$.
For $c\in C$, both \eqref{eq:Orribile1} and \eqref{eq:Orribile2}
give that $f:A_{\le a}\ra B_{\le b}$ and $g:B_{\le b}\ra C_{\le c}$
are infinitesimal, and the conclusion follows from Lemma~\ref{lem:1stComposition}.
\end{proof}

\subsection{\label{sub:Downward-directed-sets}Downward directed and segmented
sets of indices}

In this section, we study suitable classes of sets of indices where
the notion of morphism of the category $\textbf{Ind}$ simplifies.
\begin{defn}
Let $\mathbb{I}=(I,\leq,\mathcal{I})$ be a set of indices, then we
say that
\begin{enumerate}[leftmargin=*,label=(\roman*),align=left]
\item $\mathbb{I}$ is \emph{segmented} if $\forall A\in\mathcal{I}\,\exists a:\ (\emptyset,a]\subseteq A$;
\item $\mathbb{I}$ is \emph{downward directed} if $(I,\le)$ is downward
directed, i.e.~for every $a$, $b\in I$ there exists $c\in I$ such
that $c\leq a,$ $c\leq b$.
\end{enumerate}

\noindent Moreover, if $\mathbb{I}$ is downward directed, we call
\emph{canonical} \emph{set of indices} generated by $\mathbb{I}$,
and we denote it by $\overline{\mathbb{I}}$, the set of indices $\overline{\mathbb{I}}=(I,\leq,\mathcal{S}_{I}),$
where 
\[
\mathcal{S}_{I}:=\{(\emptyset,a]\mid a\in I\}\cup\{I\}.
\]

\noindent Since $(I,\leq)$ is downward directed, it is immediate
to prove that $\overline{\mathbb{I}}$ is a set of indices. \end{defn}
\begin{example}
~
\begin{enumerate}[leftmargin=*,label=(\roman*),align=left]
\item If $\mathbb{I}=\mathbb{I}^{\srm}$ then $\overline{\mathbb{I}}=\mathbb{I}$.
\item If $\mathbb{I}=\widetilde{\mathbb{I}}^{\erm}$ then then $\overline{\mathbb{I}}\neq\mathbb{I}.$ 
\end{enumerate}
\end{example}
As mentioned above, the notion of morphism is simplified when we work
with this type of sets of indices.
\begin{thm}
\label{thm:downdir}Let $\mathbb{I}_{1},\mathbb{I}_{2}$ be sets of
indices and let $f:I_{2}\ra I_{1}$ be a map. Let us assume that $\mathbb{I}_{1}$
is segmented and $\mathbb{I}_{2}$ is downward directed. Then the
following conditions are equivalent:
\begin{enumerate}[leftmargin=*,label=(\roman*),align=left]
\item \label{enu:downdir1}$f:\mathbb{I}_{1}\longrightarrow\mathbb{I}_{2}$
is a morphism of sets of indices;
\item \label{enu:downdir2}$\forall a\in I_{1}\,\exists b\in I_{2}:\ f((\emptyset,b])\subseteq(\emptyset,a]$;
\item \label{enu:downdir3}$\forall a\in I_{1}\,\forall b\in I_{2}:\ f:\left(I_{1}\right)_{\leq a}\longrightarrow\left(I_{2}\right)_{\leq b}$
is infinitesimal.
\end{enumerate}
\end{thm}
\begin{proof}
To prove that \ref{enu:downdir1} entails \ref{enu:downdir2}, let
$f:\mathbb{I}_{1}\longrightarrow\mathbb{I}_{2}$ be a morphism and
let $a\in I_{1}$. Setting $A=I_{1}$ in the definition of morphism,
we get the existence of $B\in\mathcal{I}_{2}$ such that $f:A_{\le a}\ra B_{\le\bar{b}}$
is infinitesimal for each $\bar{b}\in B$. Take any $\bar{b}\in B\ne\emptyset$.
Setting $\alpha=a$ in the definition of infinitesimal (Def.~\ref{def:MorphInd}),
we get the existence of $b\in B_{\le\bar{b}}\subseteq I_{2}$ such
that $f(\eps_{2})\in(\emptyset,a]$ for all $\eps_{2}\in(\emptyset,b]$,
which is our conclusion. 

\noindent To prove that \ref{enu:downdir2} entails \ref{enu:downdir3},
let $a\in I_{1},\, b\in I_{2}$ and let $\overline{b}\in I_{2}$ be
such that $f(\emptyset,\overline{b}]\subseteq(\emptyset,a]$. Let
$\alpha\in(\emptyset,a]$ and let $\widetilde{b}\in B_{\leq\overline{b}}$
be such that $f(\emptyset,\widetilde{b}]\subseteq(\emptyset,\alpha]$.
Since $(I_{2},\le)$ is downward directed, we can find $\beta\in I_{2}$
such that $\beta\leq b,\,\beta\leq\widetilde{b}.$ By construction,
$f(\emptyset,\beta]\subseteq(\emptyset,\alpha]$ and $(\emptyset,\beta]\subseteq(\emptyset,b]=\left(I_{2}\right)_{\leq b}.$
Therefore $f:\left(I_{1}\right)_{\leq a}\longrightarrow\left(I_{2}\right)_{\leq b}$
is infinitesimal.

\noindent To prove that \ref{enu:downdir3} entails \ref{enu:downdir1},
assume that $a\in A\in\mathcal{I}_{1}$. Set $B:=I_{2}$ and take
any $b\in B$. By \ref{enu:downdir3} we obtain that $f:\left(I_{1}\right)_{\le a}\ra\left(I_{2}\right)_{\le b}$
is infinitesimal. Therefore, for each $\alpha\le a$ there exists
$\widetilde{\beta}\leq b$ such that we have $f(\eps_{2})\le\alpha$
for every $\eps_{2}\leq\widetilde{\beta}$. small, let's say for $\eps_{2}\le\tilde{\beta}\le b$.
But $\mathbb{I}_{1}$ is segmented, so there exists $a'$ such that
$(\emptyset,a']\subseteq A$. Once again from \ref{enu:downdir3}
we also have that $f:\left(I_{1}\right)_{\le a'}\ra\left(I_{2}\right)_{\le b}$
is infinitesimal. Hence for some $\bar{\beta}\le b$ we have $f(\eps_{2})\le a'$
for each $\eps_{2}\le\bar{\beta}$. Since $(I_{2},\le)$ is downward
directed, we can find $\beta\in I_{2}=B$ such that $\beta\le\tilde{\beta}$
and $\beta\le\bar{\beta}$. Therefore, for each $\eps_{2}\le\beta$
we have both $f(\eps_{2})\le\alpha$ and $f(\eps_{2})\in(\emptyset,a']\subseteq A$.
This proves that $f:A_{\le a}\ra B_{\le b}$ is infinitesimal, which
completes the proof.\end{proof}
\begin{thm}
\label{thm:For-every-downward} Every segmented downward directed
set of indices $\mathbb{I}$ is isomorphic to $\bar{\mathbb{I}}$
in the category \textbf{\emph{Ind}}.\end{thm}
\begin{proof}
It suffices to consider the identity $1_{I}:i\in I\mapsto i\in I$,
which is a morphism $1_{I}\in\textbf{Ind}(\mathbb{I},\bar{\mathbb{I}})\cap\textbf{Ind}(\bar{\mathbb{I}},\mathbb{I})$
because of condition \ref{enu:downdir2} of Thm \ref{thm:downdir}. 
\end{proof}
\noindent Therefore, up to isomorphism, the only segmented downward
directed set of indices having $(I,\le)$ as underlying pre-ordered
set is $\overline{\mathbb{I}}.$

\section{Asymptotic Gauge Colombeau Type Algebras}

\subsection{Asymptotic Gauges}

In \cite{GiLu15}, we introduced the notion of asymptotic gauge. The
idea was to use it as an asymptotic scale that generalizes the role
of the polynomial family $(\eps^{n})_{\eps\in(0,1],n\in\N}$ in classical
constructions of Colombeau algebras. We recall the notations and notions
from \cite{GiLu15} that we will use in the present work. For all
the proofs, we refer to \cite{GiLu15}.
\begin{defn}
\label{def:asymptoticGauge}Let $\mathbb{I}=(I,\le,\mathcal{I})$
be a set of indices. All big-Os in this definition have to be meant
as $O_{\mathcal{I}}$ in $\mathbb{I}$ (see Def. \ref{def:2ndBigOh}).
We say that $\mathcal{B}$ is an \emph{asymptotic gauge on $\mathbb{I}$}
(briefly: AG on $\mathbb{I}$) if
\begin{enumerate}[leftmargin=*,label=(\roman*),align=left ]
\item \label{enu:AGnets}$\mathcal{B}\subseteq\R^{I}$;
\item \label{enu:AGinfinite}$\exists i\in\mathcal{B}:\ \lim_{\mathbb{I}}i=\infty$;
\item \label{enu:AGProduct}$\forall i,j\in\mathcal{B}\,\exists p\in\mathcal{B}:\ i\cdot j=O(p)$;
\item \label{enu:AGproductScalar}$\forall i\in\mathcal{B}\,\forall r\in\R\,\exists\sigma\in\mathcal{B}:\ r\cdot i=O(\sigma)$;
\item \label{enu:AGabsSum}$\forall i,j\in\B\,\exists s\in\B:\ s>_{\mathbb{I}}0\ ,\ |i|+|j|=O(s)$.
\end{enumerate}
Let $\B$ be an AG on the set of indices $\mathbb{I}=(I,\le,\mathcal{I})$.
The set of \emph{moderate nets generated by} \emph{$\B$} is 

\[
\R_{M}(\B):=\left\{ x\in\R^{I}\mid\exists b\in\B:\ x_{\eps}=O(b_{\eps})\right\} .
\]

\end{defn}
Let us observe that $\R_{M}(\B)$ is an AG, and that $\R_{M}(\R_{M}(\B))=\R_{M}(\B).$
Every asymptotic gauge formalizes a notion of ``growth condition''.
We can hence use an asymptotic gauge $\B$ to define moderate nets.
We can also use the reciprocals of nets taken from another asymptotic
gauge $\mathcal{Z}$ to define negligible nets. From this point of
view, it is natural to introduce the following:
\begin{defn}
Let $\Omega\subseteq\mathbb{R}^{n}$ be an open set, let $\mathcal{B}$,
$\mathcal{Z}$ be AG on the same set of indices $\mathbb{I}=(I,\leq,\mathcal{I})$.
The \emph{set of $\B$-moderate nets }is
\[
\mathcal{E}_{M}(\B,\Omega):=\{u\in\mathcal{C}^{\infty}(\Omega)^{I}\mid\forall K\Subset\Omega\,\forall\alpha\in\mathbb{N}^{n}\exists b\in\B:\ \sup\limits _{x\in K}|\partial^{\alpha}u_{\eps}(x)|=O(b_{\eps})\}.
\]

\noindent The \emph{set of $\mathcal{Z}$-negligible nets} is
\begin{equation}
\mathcal{N}(\mathcal{Z},\Omega):=\{u\in\mathcal{C}^{\infty}(\Omega)^{I}\mid\forall K\Subset\Omega\,\forall\alpha\in\mathbb{N}^{n}\forall z\in\mathcal{Z}_{>0}:\ \sup\limits _{x\in K}|\partial^{\alpha}u_{\eps}(x)|=O(z_{\eps}^{-1})\}.\label{eq:DefNegligible}
\end{equation}

\end{defn}
In \cite{GiLu15}, we proved that if $\R_{M}(\B)\subseteq\R_{M}(\mathcal{Z})$
then the quotient $\mathcal{E}_{M}(\B,\Omega)/\mathord{\mathcal{N}(\mathcal{Z},\Omega)}$
is an algebra. When this happens, we will use the following:
\begin{defn}
\label{def:Colombeau-AG-Algebra}Let $\B,\mathcal{Z}$ be AG on the
same set of indices $\mathbb{I}=(I,\leq,\mathcal{I})$ such that $\R_{M}(\B)\subseteq\R_{M}(\mathcal{Z})$.
The \emph{Colombeau AG algebra generated by $\B$ and $\mathcal{Z}$}
is the quotient 
\[
\mathcal{G}(\B,\mathcal{Z}):=\mathcal{E}_{M}(\B,\Omega)/\mathord{\mathcal{N}(\mathcal{Z},\Omega)}.
\]

\noindent We will use the notation $\GBZ$ to emphasize the dependence
on the open set $\Omega$.
\end{defn}
Morphisms between sets of indices can be used to construct asymptotic
gauges, as the following theorem shows.
\begin{thm}
Let $\mathcal{B}$ be an asymptotic gauge on the set of indices $\mathbb{I}_{1}$
and let $f:\mathbb{I}_{1}\longrightarrow\mathbb{I}_{2}$ be a morphism.
Then
\[
\mathcal{B}\circ f=\{b\circ f\mid b\in\mathcal{B}\}
\]
is an asymptotic gauge on $\mathbb{I}_{2}$.\end{thm}
\begin{proof}
All the defining properties of an asymptotic gauge for $\mathcal{B}\circ f$
can be derived from Cor.~\ref{cor:morfind}. For example, let us
prove that $\forall i,j\in\B\circ f\,\exists s\in\B\circ f:\ s>_{\mathbb{I}}0\ ,\ |i|+|j|=O(s).$
Let $i$, $j\in\B\circ f$ and let $i=b_{1}\circ f,\, j=b_{2}\circ f$.
Let $b_{3}\in\B$ be such that $|b_{1}|+|b_{2}|=O(b_{3}).$ Then by
Cor.~\ref{cor:morfind} we deduce that $|b_{1}\circ f|+|b_{2}\circ f|=O(b_{3}\circ f).$
Setting $s=b_{3}\circ f$ we therefore have that $|i|+|j|=O(s).$
\end{proof}

\section{\label{sec:The-categories-Ag21}The categories Ag$_{2}$ and Ag$_{1}$}

We want to prove that the Colombeau AG algebra construction of Def.~\ref{def:Colombeau-AG-Algebra}
is functorial in the pair $(\B,\mathcal{Z})$ of AG. In proving this
result, the following category arises naturally:
\begin{defn}
\label{def:AG2}We set
\begin{enumerate}[leftmargin=*,label=(\roman*),align=left ]
\item $(\mathcal{B},\mathcal{Z})\in\textsc{Ag}_{2}$ if $\mathcal{B}$,
$\mathcal{Z}$ are AG on some set of indices $\mathbb{I}$ and $\R_{M}(\B)\subseteq\R_{M}(\mathcal{Z})$.
\item \label{enu:defAG2-morph}Let $(\mathcal{B}_{1},\mathcal{Z}_{1})$,
$(\mathcal{B}_{2},\mathcal{Z}_{2})\in\textsc{Ag}_{2}$ be pairs of
AG on the sets of indices resp.~$\mathbb{I}_{1}$, $\mathbb{I}_{2}$.
We say that $f\in\textsc{Ag}_{2}\left((\B_{1},\mathcal{Z}_{1}),(\B_{2},\mathcal{Z}_{2})\right)$
is a \emph{morphism of pairs of AG} if $f\in\textbf{Ind}(\mathbb{I}_{1},\mathbb{I}_{2})$,
$\mathbb{R}_{M}(\mathcal{B}_{1}\circ f)\subseteq\mathbb{R}_{M}(\mathcal{B}_{2})$
and $\R_{M}(\mathcal{Z}_{2})\subseteq\R_{M}(\mathcal{Z}_{1}\circ f)$.
\end{enumerate}
\end{defn}
\begin{thm}
$\textsc{Ag}_{2}$ with set-theoretical composition and identity is
a category.\end{thm}
\begin{proof}
It is sufficient to consider the composition. Let $f\in\textsc{Ag}_{2}\left((\B_{1},\mathcal{Z}_{1}),(\B_{2},\mathcal{Z}_{2})\right)$
and $g\in\textsc{Ag}_{2}\left((\B_{2},\mathcal{Z}_{2}),(\B_{3},\mathcal{Z}_{3})\right)$.
By definition, $f\in\textbf{Ind}(\mathbb{I}_{1},\mathbb{I}_{2})$
and $g\in\textbf{Ind}(\mathbb{I}_{2},\mathbb{I}_{3})$, therefore
$f\circ g\in\textbf{Ind}(\mathbb{I}_{1},\mathbb{I}_{3})$ by Thm.~\ref{thm:IndCategory}.
Moreover, 
\[
\mathbb{R}_{M}(\mathcal{B}_{1}\circ(f\circ g))=\{b\circ f\circ g\mid b\in\mathcal{B}_{1}\}\subseteq\{b\circ g\mid b\in\mathcal{B}_{2}\}=\mathcal{B}_{2}\circ g,
\]
since $\mathbb{R}_{M}(\mathcal{B}_{1}\circ f)\subseteq\mathbb{R}_{M}(\mathcal{B}_{2})$.
But $\mathcal{B}_{2}\circ g\subseteq\mathbb{R}_{M}(\mathcal{B}_{2}\circ g)\subseteq\R_{M}(\mathcal{B}_{3})$,
from which the first part of the conclusion follows. To prove the
second part of the conclusion we notice that, as $\mathbb{R}_{M}(\mathcal{Z}_{2})\subseteq\mathbb{R}_{M}(\mathcal{Z}_{1}\circ f)$,
we have that 
\[
\mathbb{R}_{M}(\mathcal{Z}_{2}\circ g)\subseteq\mathbb{R}_{M}((\mathcal{Z}_{1}\circ f)\circ g)=\mathbb{R}_{M}(\mathcal{Z}_{1}\circ(f\circ g)),
\]
and the thesis follows since $\mathbb{R}_{M}(\mathcal{Z}_{3})\subseteq\mathbb{R}_{M}(\mathcal{Z}_{2}\circ g)$
by hypothesis.
\end{proof}
The generalization with two AG is a relatively new step in considering
Colombeau like algebras. It is therefore natural to consider also
the following
\begin{defn}
\label{def:AG1}We say that $\B\in\textsc{Ag}_{1}$ if $(\B,\B)\in\textsc{Ag}_{2}$.
We set $f\in\textsc{Ag}_{1}(\B_{1},\B_{2})$ if $f\in\textsc{Ag}_{2}((\B_{1},\B_{1}),(\B_{2},\B_{2}))$.
We call such an $f$ a \emph{morphism of AG}. In this case, Def.~\ref{def:AG2}~\ref{enu:defAG2-morph}
becomes $\R_{M}(\B_{1}\circ f)=\R_{M}(\B_{2})$. 
\end{defn}
Of course $\textsc{Ag}_{1}$ is embedded into $\textsc{Ag}_{2}$ by
means of $\B\mapsto(\B,\B)$ and of the identity on arrows. By an
innocuous abuse of language, we can hence say that $\textsc{Ag}_{1}$
is a subcategory of $\textsc{Ag}_{2}$.
\begin{example}
~
\begin{enumerate}[leftmargin=*,label=(\roman*),align=left ]
\item  Let $\mathbb{I}=\mathbb{I}^{\srm}$, let $\mathcal{B}_{1}=\{\left(\varepsilon^{-n}\right)\mid n\in\mathbb{N}\}$,
$\mathcal{B}_{2}=\{\left(\varepsilon^{-2n}\right)\mid n\in\mathbb{N}\}$.
Then $f$, $g:\mathbb{I\longrightarrow\mathbb{I}}$ such that $f(\varepsilon)=\varepsilon^{2}$
and $g(\varepsilon)=\sqrt{\varepsilon}$ induce morphisms $\mathcal{B}_{1}\overset{f}{\longrightarrow}\mathcal{B}_{2}$
and $\mathcal{B}_{2}\overset{g}{\longrightarrow}\mathcal{B}_{1}$.
Clearly $f\circ g=1_{\mathcal{B}_{2}}$ and $g\circ f=1_{\mathcal{B}_{1}}$,
therefore $\mathcal{B}_{1}$ and $\mathcal{B}_{2}$ are isomorphic.
\item Let $\mathbb{I}_{1}=\mathbb{I}^{\srm},$ $\mathbb{I}_{2}=\overline{\mathbb{N}}$
(see Example~\ref{exa:morphSetsOfIndices} \ref{enu:Nbar}), let
$\mathcal{B}_{1}=\{(\varepsilon^{-n})\mid n\in\mathbb{N}\}$, $\mathcal{B}_{2}=\{(n^{m})_{n}\mid m\in\mathbb{N}\}$.
Then $f:\mathbb{I}_{1}\longrightarrow\mathbb{I}_{2}$ such that $f(n)=\frac{1}{n+1}$
for every $n\in\mathbb{N}$ induces a morphism $\mathcal{B}_{1}\overset{f}{\longrightarrow}\mathcal{B}_{2}$
and $g:\mathbb{I}_{2}\longrightarrow\mathbb{I}_{1}$ such that $g(\varepsilon)=\lfloor\frac{1}{\eps}\rfloor$
for every $\varepsilon\in(0,1]$ induces a morphism $\mathcal{B}_{2}\overset{g}{\longrightarrow}\mathcal{B}_{1}$.
\item Let $o:\overline{\mathbb{N}}\longrightarrow\widetilde{\mathbb{I}}^{\erm}$
be the morphism given by the maps $o$ that maps every $\varphi\in\mathcal{A}_{0}$
to the order $o(\varphi)$ of $\varphi$. Let $\mathcal{B}_{1}=\{(b_{n}^{m})_{n}\mid m\in\mathbb{N}\}$
where $b_{n}=\frac{1}{n+1}$ for every $n\in\mathbb{N}$ and let $\mathcal{B}_{2}=\{(b_{\varphi}^{n})_{\phi}\mid n\in\mathbb{N}\}$,
where $b_{\varphi}=\frac{1}{o(\varphi)+1}$ for every $\varphi\in\mathcal{A}_{0}$.
Then $\mathcal{B}_{1}\overset{o}{\longrightarrow}\mathcal{B}_{2}$
is a morphism. 
\item Set $f(\eps)=\eps+\eps^{2}\cdot\sin\left(\frac{1}{\eps}\right)$ for
$\eps\in(0,1]$ and $\B^{\srm}:=\left\{ \eps^{-a}\mid a\in\R_{>0}\right\} $.
Then $\eps-\eps^{2}\le f(\eps)\le\eps+\eps^{2}$, and this implies
$f\in\textsc{Ag}_{1}(\B^{\srm},\B^{\srm})$. Let us note that $f$
is not invertible in any neighbourhood of $0$ so that it is not an
isomorphism of AG.
\end{enumerate}
\end{example}
In \cite{GiLu15}, we defined two asymptotic gauges $\mathcal{B}_{1}$,
$\mathcal{B}_{2}$ to be equivalent if and only if $\R_{M}(\B_{1})=\R_{M}(\B_{2}).$
Within the present categorical framework, this definition is motivated
by the following result.
\begin{thm}
Let $\mathcal{B}$ be an asymptotic gauge on $\mathbb{I}$. Then $\mathcal{B}$
is isomorphic to $\mathbb{R}_{M}(\mathcal{B})$.\end{thm}
\begin{proof}
It is sufficient to observe that, by definition, $\mathcal{B}\overset{1_{I}}{\longrightarrow}\mathbb{R}_{M}(\mathcal{B})$
is a morphism and, since $\mathbb{R}_{M}(\mathbb{R}_{M}(\mathcal{B}))=\mathbb{\mathbb{R}}_{M}(\mathcal{B})$,
also $\mathbb{R}_{M}(\mathcal{B})\overset{1_{I}}{\longrightarrow}\mathcal{B}$
is a morphism. 
\end{proof}
In particular, it follows that for every two asymptotic gauges $\mathcal{B}_{1}$,
$\mathcal{B}_{2}$ defined on the same set of indices $\mathbb{I}$,
we have that if $\mathcal{B}_{1}$ is equivalent to $\mathcal{B}_{2}$
then they are isomorphic. Conversely, if $f\in\textsc{Ag}_{1}(\mathcal{B}_{1},\mathcal{B}_{2})$
is an isomorphism, then $\R_{M}(\mathcal{B}_{1})=\R_{M}(\mathcal{B}_{1}\circ f\circ f^{-1})=\R_{M}(\mathcal{B}_{2}\circ f^{-1})=\R_{M}(\mathcal{B}_{1})$,
and hence $\R_{M}(\mathcal{B}_{1})=\R_{M}(\mathcal{B}_{2}\circ f^{-1})$.
Analogously, $\R_{M}(\mathcal{B}_{2})=\R_{M}(\mathcal{B}_{1}\circ f)$.
In particular, the identity $1_{I}\in\textsc{Ag}_{1}(\mathcal{B}_{1},\mathcal{B}_{2})\cap\textsc{Ag}_{1}(\mathcal{B}_{2},\mathcal{B}_{1})$
if and only if these AG are equivalent. For example $\left\{ \left(\eps^{-a}\right)\mid a\in\R_{>0}\right\} $
and $\left\{ \left(\eps^{-n}\right)\mid n\in\N\right\} $ are equivalent.
Nevertheless, it is not difficult to prove that not all isomorphic
AG on the same set of indices are equivalent. To prove this result,
we need to recall (see \cite[Def.~36]{GiLu15}) that an AG $\B$ is
called \emph{principal} if there exists a \emph{generator} $b\in\B$
such that $\R_{M}(\text{AG}(b))=\R_{M}(\B)$, where $\text{AG}(b):=\left\{ b^{m}\mid m\in\N\right\} $.
\begin{thm}
\label{thm:Ordering of AG}For every principal AG $\mathcal{B}_{1}$,
$\mathcal{B}_{2}$ on $\mathbb{I}^{\srm}$, if $\mathbb{R}_{M}(\mathcal{B}_{1})\subsetneq\mathbb{R}_{M}(\mathcal{B}_{2})$
then there exists a principal AG $\mathcal{B}_{3}$ such that $\mathbb{R}_{M}(\mathcal{B}_{1})\subsetneq\mathbb{R}_{M}(\mathcal{B}_{3})\subsetneq\mathbb{R}_{M}(\mathcal{B}_{2}).$\end{thm}
\begin{proof}
Let $\mathcal{B}_{1}=AG(b_{1})$ and $\mathcal{B}_{2}=AG(b_{2}).$
Without loss of generality we can assume that $b_{1,\varepsilon},b_{2,\varepsilon}>1$
for every $\varepsilon\in(0,1]$. Moreover, as $\mathbb{R}_{M}(\mathcal{B}_{1})\subsetneq\mathbb{R}_{M}(\mathcal{B}_{2})$,
we have $b_{1}=O_{\mathcal{I}^{\srm}}(b_{2}).$ So, without loss of
generality, we can also assume that $b_{1,\varepsilon}<b_{2,\varepsilon}$
for every $\varepsilon\in(0,1]$. Since $\mathbb{R}_{M}(\mathcal{B}_{1})\subsetneq\mathbb{R}_{M}(\mathcal{B}_{2}),$
we have that $b_{2}\notin\mathbb{R}_{M}(b_{1}),$ namely that 
\[
\forall n\in\mathbb{N}\,\forall\varepsilon\in(0,1]\,\exists\delta<\varepsilon:\ n\cdot b_{1,\delta}^{n}<b_{2,\delta}.
\]
Now, we let $b_{3}\in\mathbb{R}^{(0,1]}$ be a net such that $b_{1,\varepsilon}\leq b_{3,\varepsilon}\leq b_{2,\varepsilon}$
for $\eps$ small, and $\forall n\in\mathbb{N}$ 
\[
b_{3,\overline{\varepsilon}_{n}}=\begin{cases}
b_{1,\overline{\varepsilon}_{n}} & \text{if \ensuremath{n}is odd};\\
b_{2,\overline{\varepsilon}_{n}} & \text{if \ensuremath{n}is even,}
\end{cases}
\]
where $\overline{\varepsilon}_{1}=1$ and 
\[
\overline{\varepsilon}_{n}\in\left\{ \varepsilon<\min\left\{ \frac{1}{n},\overline{\varepsilon}_{n-1}\right\} \mid n\cdot b_{1,\varepsilon}^{n}<b_{2,\varepsilon}\right\} 
\]
for every $n\geq2$. Since $b_{1,\varepsilon}\leq b_{3,\varepsilon}\leq b_{2,\varepsilon}$
for $\eps$ small, we have that that $\mathbb{R}_{M}(\mathcal{B}_{1})\subseteq\mathbb{R}_{M}(AG(b_{3}))\subseteq\mathbb{R}_{M}(\mathcal{B}_{2})$.
Let us prove that the reverse inclusions do not hold. To prove that,
let us assume, by contradiction, that $\mathbb{R}_{M}(AG(b_{3}))\subseteq\mathbb{R}_{M}(\mathcal{B}_{1}).$
In particular, there exists $k\in\mathbb{N}$ such that $b_{3}=O(b_{1}^{k}),$
namely there exists $\varepsilon\in(0,1],\, H\in\mathbb{R}_{>0}$
such that $b_{3,\delta}\leq H\cdot b_{1,\delta}^{k}$ for all $\delta<\varepsilon$.
Set 
\[
N:=\min\left\{ n\in\mathbb{N}\mid n\mbox{ is even and }n\geq\max\left\{ \left\lceil H\right\rceil ,\left\lceil \frac{1}{\varepsilon}\right\rceil ,k\right\} \right\} .
\]
We have 
\[
N\cdot b_{1,\overline{\varepsilon}_{N}}^{N}<b_{2,\overline{\varepsilon}_{N}}=b_{3,\overline{\varepsilon}_{N}}\leq H\cdot b_{1,\overline{\varepsilon}_{N}}^{k}\leq N\cdot b_{1,\overline{\varepsilon}_{N}}^{N},
\]
which is absurd. To prove that $\mathbb{R}_{M}(AG(b_{3}))\subsetneq\mathbb{R}_{M}(\mathcal{B}_{2})$,
we proceed in a similar way. Let us assume, by contradiction, that
$\mathbb{R}_{M}(\mathcal{B}_{2})\subseteq\mathbb{R}_{M}(AG(b_{3}))$.
Let $k\in\mathbb{N}$ be such that $b_{2}=O(b_{3}^{k})$. Let $\varepsilon\in(0,1],\, H\in\mathbb{R}_{>0}$
be such that $b_{2,\delta}\leq H\cdot b_{3,\delta}^{k}$ for all $\delta<\varepsilon$.
Set 
\[
N:=\min\left\{ n\in\mathbb{N}\mid n\mbox{ is odd and }n\geq\max\left\{ \left\lceil H\right\rceil ,\left\lceil \frac{1}{\varepsilon}\right\rceil ,k\right\} \right\} .
\]
We have 
\[
b_{2,\overline{\varepsilon}_{N}}\leq H\cdot b_{3,\overline{\varepsilon}_{N}}^{k}\leq N\cdot b_{3,\overline{\varepsilon}_{N}}^{N}=N\cdot b_{1,\overline{\varepsilon}_{N}}^{N}<b_{2,\overline{\varepsilon}_{N}},
\]
which is absurd.\end{proof}
\begin{cor}
\label{cor:Ordering of AG, part 2}For every principal AG $\mathcal{B}_{1},\mathcal{B}_{2}$
on $\mathbb{I}^{\srm}$, if $\mathbb{R}_{M}(\mathcal{B}_{1})\subsetneq\mathbb{R}_{M}(\mathcal{B}_{2})$
then there exists an infinite sequence $\langle\mathcal{A}_{i}\mid i\in\mathbb{Z}\rangle$
of principal AG on $\mathbb{I}^{\srm}$ such that 
\[
\mathbb{R}_{M}(\mathcal{B}_{1})\subsetneq\dots\subsetneq\mathbb{R}_{M}(\mathcal{A}_{-1})\subsetneq\mathbb{R}_{M}(\mathcal{A}_{0})\subsetneq\mathbb{R}_{M}(\mathcal{A}_{1})\subsetneq\dots\subsetneq\mathbb{R}_{M}(\mathcal{\mathcal{B}}_{2}).
\]
\end{cor}
\begin{proof}
This is an immediate consequence of Theorem \ref{thm:Ordering of AG}.
\end{proof}
In particular, if we let $\B_{\text{pol}}:=\left\{ (\eps^{-n})\mid n\in\N\right\} $
and $\B_{\text{exp}}:=\left\{ \left(e^{n/\eps}\right)\mid n\in\N\right\} $,
by Corollary \ref{cor:Ordering of AG, part 2} we have that there
are infinitely many principal non equivalent AG between $\B_{\text{pol}}$
and $\B_{\text{exp}}$. However, as we will show in Section \ref{sec:An-unexpected-isomorphism},
$\B_{\text{pol}}:=\left\{ (\eps^{-n})\mid n\in\N\right\} $ and $\B_{\text{exp}}:=\left\{ \left(e^{n/\eps}\right)\mid n\in\N\right\} $
are isomorphic, and this shows that not all isomorphic AG are equivalent.

In \cite{GiLu15}, we proved that $\R_{M}(\B)$ is the minimal (with
respect to inclusion) asymptotically closed solid ring containing
the AG $\B$. Therefore, we deduce that, modulo isomorphism, all the
objects in a skeleton subcategory of $\mathbf{\textsc{Ag}}_{1}$ are
asymptotically closed solid rings.

In \cite{GiLu15}, we introduced the notion of ``exponential of an
AG'', which was crucial to study linear ODE's with generalized constant
coefficients. We recall its definition.
\begin{defn}
\label{def:expB}Let $\B$ be an AG, and let $\mu:\R\ra\R_{\ge0}$
be a non decreasing function such that
\begin{align}
\lim_{x\to+\infty}\mu(x) & =+\infty;\nonumber \\
\forall b\in\B\,\exists c\in\B:\  & \mu(b_{\eps})^{2}<_{\mathbb{I}}\mu(c_{\eps}).\label{eq:muSquare}
\end{align}
We set $\mu(\B):=\left\{ \left(\mu(H\cdot b_{\eps})\right)_{\eps}\mid H\in\R_{>0},b\in\B\right\} $.
In particular,
\[
e^{\B}:=\left\{ e^{H\cdot b}\mid H\in\R_{>0},b\in\B\right\} 
\]
is called the \emph{exponential of $\B$}.
\end{defn}
\noindent The following results will be needed to prove Thm.~\ref{thm:expFunctor}.
\begin{lem}
\label{lem:muOfB}In the hypotheses of Def.~\ref{def:expB}, we have
that $\mu(\B)$ is an AG.\end{lem}
\begin{proof}
Def.~\ref{def:asymptoticGauge} \ref{enu:AGinfinite} and \cite[Lemma~18]{GiLu15}
imply the existence of $i\in\B$ such that $\lim_{\mathbb{I}}i=+\infty$.
Therefore, our assumptions yield $\lim_{\eps\in\mathbb{I}}\mu(i_{\eps})=+\infty$,
which proves Def.~\ref{def:asymptoticGauge} \ref{enu:AGinfinite}
for $\mu(\B)$. The asymptotic closure with respect to sum of absolute
values follows from monotonicity of $\mu$ and the inequality $\mu(H\cdot i_{\eps})+\mu(K\cdot j_{\eps})\le2\cdot\mu\left(H\cdot\left|i_{\eps}\right|+K\cdot\left|j_{\eps}\right|\right)$.
The asymptotic closure with respect to product follows from the inequality
\[
\mu(H\cdot i_{\eps})\cdot\mu(K\cdot j_{\eps})\le\mu\left(H\cdot\left|i_{\eps}\right|+K\cdot\left|j_{\eps}\right|\right)^{2}
\]
 and from assumption \eqref{eq:muSquare}.\end{proof}
\begin{lem}
\label{lem:big-OAndIncreas}Let $\mathbb{I}=(I,\le,\mathcal{I})$
be a set of indices, and let $x$, $y$, $z\in\R^{I}$. Let $\mu:\R\ra\R_{\ge0}$
be a non decreasing function. Then $x_{\eps}=O\left[\mu(y_{\eps})\right]$
and $y<_{\mathbb{I}}z$ imply $x_{\eps}=O\left[\mu(z_{\eps})\right]$.\end{lem}
\begin{proof}
From the assumptions we get
\begin{align*}
 & \exists A\in\mathcal{I}\,\forall a\in A:\ x_{\eps}=O_{aA}\left[\mu(y_{\eps})\right];\\
 & \exists B\in\mathcal{I}\,\forall b\in B\,\exists\eps_{0}\le b\,\forall\eps\in B_{\le\eps_{0}}:\ y_{\eps}<z_{\eps}.
\end{align*}
Def.~\ref{def:setOfIndices} \ref{enu:SoI-Filter} implies the existence
of $D\in\mathcal{I}$ such that $D\subseteq A\cap B$. For $d\in D$,
\cite[Thm.~2.8 (x)]{GiNi14} yields $x_{\eps}=O_{dD}\left[\mu(y_{\eps})\right]$,
and therefore, for suitable $H\in\R_{>0}$ and $\eps_{0}\le d$, $\eps_{1}\le d$,
we can write $|x_{\eps}|\le H\left|\mu(y_{\eps})\right|=H\mu(y_{\eps})$
for each $\eps\in D_{\le\eps_{1}}$ and $y_{\eps}<z_{\eps}$ for each
$\eps\in D_{\le\eps_{0}}$. Since $(\emptyset,d]=I_{\le d}$ is directed,
we can find $\bar{\eps}\le d$, $\eps_{0}$, $\eps_{1}$. Therefore,
for each $\eps\le\bar{\eps}$ we have $\left|x_{\eps}\right|\le H\mu(y_{\eps})\le H\mu\left(z_{\eps}\right)$
because $\mu$ is non decreasing. This proves our conclusion.\end{proof}
\begin{cor}
\textup{\label{cor:exp}Let $\mathcal{B}_{1}$, $\mathcal{B}_{2}$
be AG on the same set of indices $\mathbb{I}$, and let }$\mu:\R\ra\R_{\ge0}$\textup{
verify the assumptions of Def.~\ref{def:expB}. Then }$\mathbb{R}_{M}(\mathcal{B}_{1})\subseteq\mathbb{R}_{M}(\mathcal{B}_{2})$
\textup{implies} $\mathbb{R}_{M}(\mu(\B_{1}))\subseteq\mathbb{R}_{M}(\mu(\B_{2}))$.\end{cor}
\begin{proof}
Let $(y_{\eps}')=(\mu(H\cdot b'_{\eps}))\in\mu(\B_{1})$, with $b'\in\B_{1}$,
and let $b''\in\B_{2}$ be such that $b'<_{\mathbb{I}}b''$. Then
we have that $(y'_{\eps})<_{\mathbb{I}}(\mu(H\cdot b_{\eps}''))\in\mu\left(\B_{2}\right)$
since $\mu$ is monotone.\end{proof}
\begin{defn}
Let $\textsc{Ag}_{\le}$ be the subcategory of $\textsc{Ag}_{1}$
having the same objects of $\textsc{Ag}_{1}$, and arrows such that
$f\in\textsc{Ag}_{\le}(\mathcal{B}_{1},\mathcal{B}_{2})$ if $f\in\textsc{Ind}(\mathbb{I}_{1},\mathbb{I}_{2})$
and $\mathbb{R}_{M}\left(\mathcal{B}_{1}\circ f\right)=\mathbb{R}_{M}\left(\mathcal{B}_{2}\right)$.
Let $\mu$ verifies the assumptions of Def.~\ref{def:expB}, and
let $E_{\mu}:\textsc{Ag}_{\le}\longrightarrow\textsc{Ag}_{\le}$ be
defined on objects and maps of $\textsc{Ag}_{\le}$ as follows:
\begin{enumerate}[leftmargin=*,label=(\roman*),align=left ]
\item $E(\mathcal{B}):=\mu(\B)$ for each $\mathcal{B}\in\textsc{Ag}_{\le}$;
\item $E(f):=f$ for each $f\in\textsc{Ag}_{\le}(\mathcal{B}_{1},\mathcal{B}_{2})$.
\end{enumerate}
\end{defn}
\begin{thm}
\label{thm:expFunctor}$\textsc{Ag}_{\le}$ is a subcategory of $\textsc{Ag}_{1}$.
If $\mu$ verifies the assumptions of Def.~\ref{def:expB}, then
$E_{\mu}:\textsc{Ag}_{\le}\longrightarrow\textsc{Ag}_{\le}$ is a
functor.\end{thm}
\begin{proof}
To prove the first part, assume that $\mathbb{R}_{M}\left(\mathcal{B}_{1}\circ f\right)=\mathbb{R}_{M}\left(\mathcal{B}_{2}\right)$
and $\mathbb{R}_{M}\left(\mathcal{B}_{2}\circ g\right)=\mathbb{R}_{M}\left(\mathcal{B}_{3}\right)$.
Then if $b_{1}\circ f<_{\mathbb{I}_{2}}b_{2}$ and $b_{2}\circ g<_{\mathbb{I}_{3}}b_{3}$,
for $b_{i}\in\mathcal{B}_{i}$, then $b_{1}\circ f\circ g<_{\mathbb{I}_{3}}b_{2}\circ g<_{\mathbb{I}_{3}}b_{3}$
by Cor.~\ref{cor:morfind}. If $b_{3}<_{\mathbb{I}_{3}}b_{2}\circ g$
and $b_{2}<_{\mathbb{I}_{2}}b_{1}\circ f$ then $b_{3}<_{\mathbb{I}_{3}}b_{2}\circ g<_{\mathbb{I}_{3}}b_{1}\circ f\circ g$
once again by Cor.~\ref{cor:morfind}. This implies that $\mathbb{R}_{M}\left(\mathcal{B}_{1}\circ(f\circ g)\right)=\mathbb{R}_{M}\left(\mathcal{B}_{3}\right)$,
hence $f\circ g\in\textsc{Ag}_{\le}(\mathcal{B}_{1},\mathcal{B}_{3})$.
This and Cor.~\ref{cor:exp}~ show that $\textsc{Ag}_{\le}$ is
a category. By Cor.~\ref{cor:exp}, we also have that $\textsc{Ag}_{\le}$
is a subcategory of $\textsc{Ag}_{1}$. To show the second part, since
$E_{\mu}$ is the identity on arrows, it suffices to observe that
$\mathcal{B}$ and $\mu(\B)$ have the same set of indices for every
AG $\mathcal{B}$, and that $\mu(\B_{1})\circ f=\mu(\mathcal{B}_{1}\circ f)$.
Thus it follows by Cor.~\ref{cor:exp}~that $E_{\mu}(\mathcal{B}_{1})=\mu(\B_{1})\overset{f}{\longrightarrow}\mu(\B_{2})=E_{\mu}(\mathcal{B}_{2})$
is an arrow in $\textsc{Ag}_{\le}$ for every arrow $\mathcal{B}_{1}\overset{f}{\longrightarrow}\mathcal{B}_{2}$
in $\textsc{Ag}_{\le}$.
\end{proof}
Now we want to prove that the map $(\B,\mathcal{Z},\Omega)\mapsto\mathcal{G}(\B,\mathcal{Z},\Omega)$
is a functor. Clearly, $(\B,\mathcal{Z})\in\textsc{Ag}_{2}$, so we
need to introduce a category having as objects open sets like $\Omega$:
\begin{defn}
We denote by $\mathcal{O}\mathbb{R}^{\infty}$ the category having
as objects $\{\Omega\subseteq\mathbb{R}^{n}\mid n\in\mathbb{N},\,\Omega\,\text{open}\}$
and as morphisms $\OR(U,V):=\Coo(U,V)$.
\end{defn}
\noindent Therefore, we can now prove the following:
\begin{thm}
\label{thm:G-FunctorAlg}$\mathcal{G}:\textsc{AG}_{2}\times\left(\OR\right)^{\text{\emph{op}}}\ra\textsc{Alg}_{\R}$
is a functor, where $\textsc{Alg}_{\R}$ is the category of commutative
algebras over $\R$.\end{thm}
\begin{proof}
Let $i\in\textsc{Ag}_{2}((\B_{1},\mathcal{Z}_{1}),(\B_{2},\mathcal{Z}_{2}))$
be a morphism of pairs of AG and $h\in\Coo(\Omega_{2},\Omega_{1})$,
$\Omega_{j}$ being an open set in $\R^{n_{j}}$. The natural definition
of $\mathcal{G}(i,h)$ to get that $\mathcal{G}(i,h):\mathcal{G}(\B_{1},\mathcal{Z}_{1},\Omega_{1})\ra\mathcal{G}(\B_{2},\mathcal{Z}_{2},\Omega_{2})$
is a morphism of algebras is
\[
\mathcal{G}(i,h):[u_{\eps_{1}}]\mapsto\left[u_{i_{\eps_{2}}}\circ h\right].
\]
To prove that this definition is correct, we assume $K_{2}\Subset\Omega_{2}$
and $\alpha\in\N^{n_{2}}$. Since $(u_{\eps_{1}})\in\mathcal{E}_{M}(\B_{1},\Omega_{1})$
and $h(K_{2})=:K_{1}\Subset\Omega_{1}$, we obtain
\begin{equation}
\exists b'\in\B_{1}:\ \sup_{x\in K_{1}}\left|\partial^{\alpha}u_{\eps_{1}}(x)\right|=O(b'_{\eps_{1}})\text{ as }\eps_{1}\in\mathbb{I}_{1}.\label{eq:supOInB_1}
\end{equation}
But $b'\in\B_{1}\subseteq\R_{M}(\B_{1})$, so $b'\circ i\in\R_{M}(\B_{1}\circ i)\subseteq\R_{M}(\B_{2})$.
We can hence write $b'_{i_{\eps_{2}}}=O(b''_{\eps_{2}})$ as $\eps_{2}\in\mathbb{I}_{2}$
for a suitable $b''\in\B_{2}$. This, Cor.~\ref{cor:morfind}~\ref{enu:morphismsPreserveBigO}
and \eqref{eq:supOInB_1} yield
\[
\sup_{x\in K_{2}}\left|\partial^{\alpha}\left(u_{i_{\eps_{2}}}\circ h\right)(x)\right|=O\left(b'_{i_{\eps_{2}}}\right)=O(b''_{\eps_{2}})\text{ as }\eps_{2}\in\mathbb{I}_{2}.
\]
This shows that $\left[u_{i_{\eps_{2}}}\circ h\right]\in\mathcal{E}_{M}(\B_{2},\Omega_{2})$.
Now assume that $(u_{\eps_{1}})-(v_{\eps_{1}})\in\mathcal{N}(\mathcal{Z}_{1},\Omega_{1})$,
$K_{2}$, $\alpha$ as above, and $z\in\left(\mathcal{Z}_{2}\right)_{>0}$.
Since $\R_{M}(\mathcal{Z}_{2})\subseteq\R_{M}(\mathcal{Z}_{1}\circ i)$,
we can write $z_{\eps_{2}}=O\left(\zeta{}_{i_{\eps_{2}}}\right)$
as $\eps_{2}\in\mathbb{I}_{2}$ for a suitable $\zeta\in\left(\mathcal{Z}_{1}\right)_{>0}$,
and hence $\zeta_{i_{\eps_{2}}}^{-1}=O(z_{\eps_{2}}^{-1})$. We thus
obtain
\[
\sup_{x\in K_{1}}\left|\partial^{\alpha}u_{\eps_{1}}(x)-\partial^{\alpha}v_{\eps_{1}}(x)\right|=O(\zeta{}_{\eps_{1}}^{-1})\text{ as }\eps_{1}\in\mathbb{I}_{1},
\]
where $h(K_{2})=:K_{1}$. From this and Cor.~\ref{cor:morfind}~\ref{enu:morphismsPreserveBigO}
we obtain the conclusion
\[
\sup_{x\in K_{2}}\left|\partial^{\alpha}\left(u_{i_{\eps_{2}}}\circ h\right)(x)-\partial^{\alpha}\left(v_{i_{\eps_{2}}}\circ h\right)(x)\right|=O(z{}_{\eps_{2}}^{-1})\text{ as }\eps_{2}\in\mathbb{I}_{2}.
\]
The proof that $\mathcal{G}(i,h)$ is a morphism of $\R$-algebras
follows immediately from the pointwise definitions of the algebra
operations. The functorial properties of $\mathcal{G}$ follow directly
from the definition of $\mathcal{G}(i,h)$ and the fact that in the
domain $\textsc{AG}_{2}\times\left(\OR\right)^{\text{op}}$ composition
and identity are the corresponding set-theoretical operations.
\end{proof}
Now, let $n\in\mathbb{N}$ be fixed. Let $\mathcal{T}\R^{n}$ be the
subcategory of $\mathcal{O}\mathbb{R}^{\infty}$ having as objects
the open subsets of $\mathbb{R}^{n}$ and, as morphisms, the inclusions.
From Thm.~\ref{thm:G-FunctorAlg} we get that $\mathcal{G}(\mathcal{B},\mathcal{Z},-):\left(\mathcal{T}\R^{n}\right)^{\text{op}}\ra\textsc{Alg}_{\R}$
is a functor, i.e.~it is a presheaf. Trivially generalizing \cite{GKOS},
it is also possible to prove that $\mathcal{G}(\B,\mathcal{Z},-)$
is a sheaf of differential algebras. In particular, the following
diagram commutes
\begin{equation}
\xymatrix{\mathcal{G}(\B_{1},\mathcal{Z}_{1},\Omega_{1})\ar[r]^{\mathcal{G}(i,h)}\ar[d]_{\partial_{1}^{\alpha}} & \mathcal{G}(\B_{2},\mathcal{Z}_{2},\Omega_{2})\ar[d]^{\partial_{2}^{\alpha}}\\
\mathcal{G}(\B_{1},\mathcal{Z}_{1},\Omega_{1})\ar[r]_{\mathcal{G}(i,h)} & \mathcal{G}(\B_{2},\mathcal{Z}_{2},\Omega_{2})
}
\label{eq:diffAlgMorph}
\end{equation}
for each multi-index $\alpha\in\N^{n}$ and each inclusion $h\in\mathcal{T}\R^{n}(\Omega_{2},\Omega_{1})$.
Clearly, in \eqref{eq:diffAlgMorph}, $\partial_{k}^{\alpha}:[u_{\eps}]\in\mathcal{G}(\B_{k},\mathcal{Z}_{k},\Omega_{k})\mapsto[\partial^{\alpha}u_{\eps}]\in\mathcal{G}(\B_{k},\mathcal{Z}_{k},\Omega_{k})$.
Let us note that, in general, diagram \eqref{eq:diffAlgMorph} doesn't
commute if $h$ is an arbitrary smooth function. For this reason,
when we want to deal with differential algebras, we will always consider
$\mathcal{T}\R^{n}$ instead of the category $\OR$.

As a consequence of Thm.~\ref{thm:G-FunctorAlg}, we also have that
essentially all the constructions of Colombeau-like algebras are functorial.
For example, we can consider the set of indices $\mathbb{I}^{\srm}$
of the special algebra, the AG $\B^{\srm}:=\left\{ (\eps^{-a})\mid a\in\R_{>0}\right\} $,
and the full subcategory $\textsc{Ag}_{\mathbb{I}^{\srm}}$ of $\textsc{Ag}_{1}$
of all the AG on $\mathbb{I}^{\srm}$. Clearly, $\mathcal{G}^{\srm}(\B,\Omega):=\mathcal{G}(\B,\B,\Omega)$
is a functor $\mathcal{G}^{\srm}:\textsc{Ag}_{\mathbb{I}^{\srm}}\times\left(\OR\right)^{\text{op}}\ra\textsc{Alg}_{\R}$
which corresponds to the usual sheaf via the restriction $\mathcal{G}^{\srm}(\B^{\srm},\Omega)$
only for $\Omega\in\mathcal{T}\R^{n}$. Analogously, we can consider
$\hat{\mathcal{G}}$, $\mathcal{G}^{\erm}$, $\mathcal{G}^{\drm}$
and $\mathcal{G}^{2}$.

We also finally note that if we consider an inclusion $h\in\mathcal{T}\R^{n}(\Omega_{2},\Omega_{1})$
and a morphism of pairs of AG $i\in\textsc{Ag}_{2}((\B_{1},\mathcal{Z}_{1}),(\B_{2},\mathcal{Z}_{2}))$,
then $\mathcal{G}(i,h):\mathcal{G}(\B_{1},\mathcal{Z}_{1},\Omega_{1})\ra\mathcal{G}(\B_{2},\mathcal{Z}_{2},\Omega_{2})$
preserves all polynomial and differential operations. Of course, it
also takes generalized functions in the domain $\mathcal{G}(\B_{1},\mathcal{Z}_{1},\Omega_{1})$
into generalized functions in the codomain $\mathcal{G}(\B_{2},\mathcal{Z}_{2},\Omega_{2})$.
We can therefore state that $\mathcal{G}(i,h)$ permits to relate
differential problems framed in $\mathcal{G}(\B_{1},\mathcal{Z}_{1},\Omega_{1})$
to those framed in $\mathcal{G}(\B_{2},\mathcal{Z}_{2},\Omega_{2})$;
see also the next Thm.~\ref{thm:iso}.

\section{\label{sec:An-unexpected-isomorphism}An unexpected isomorphism}

If we set $\B_{\text{pol}}:=\left\{ (\eps^{-n})\mid n\in\N\right\} $
and $\B_{\text{exp}}:=\left\{ \left(e^{n/\eps}\right)\mid n\in\N\right\} $,
it is well known (see \cite{GKOS,GiLu15}) that an ODE like
\begin{equation}
\begin{cases}
x'(t)-[\eps^{-1}]\cdot x(t)=0;\\
x(0)=1,
\end{cases}\label{eq:ODE}
\end{equation}
has no solution in the algebra $\mathcal{G}(\B_{\text{pol}},\Omega)=\mathcal{G}^{\srm}(\Omega)$,
but it has a (unique) solution $x(t)=\left[e^{\frac{1}{\eps}t}\right]$,
$t\in\Rtil_{c}(\B_{\text{exp}})$, in $\mathcal{G}(\B_{\text{exp}},\R)$.
On the other hand, if we set $\lambda(\eps):=-\frac{1}{\log\eps}$,
for $\eps\in(0,1)$, and $\lambda(1):=1$, then we have $\lim_{\eps\to0^{+}}\lambda(\eps)=0^{+}$
and hence, by Example~\ref{exa:morphSetsOfIndices}~\ref{enu:morph-I^s},
$\lambda\in\textsc{Ind}(\mathbb{I}^{\srm},\mathbb{I}^{\srm})$ is
a morphism of set of indices. Moreover, $\left((e^{n/\eps})\circ\lambda\right)(\eps)=\eps^{-n}$
and hence $\R_{M}(\B_{\text{exp}}\circ\lambda)=\R_{M}(\B_{\text{pol}})$.
Therefore, $\lambda\in\textsc{Ag}_{1}(\B_{\text{exp}},\B_{\text{pol}})$
is a morphism of AG. Analogously, if we set $\eta(\eps):=e^{-\frac{1}{\eps}}$,
for $\eps\in(0,1)$, and $\eta(1):=1$, then we have $\lambda=\eta^{-1}$
and $\eta\in\textsc{AG}_{1}(\B_{\text{pol}},\B_{\text{exp}})$. Therefore
$\B_{\text{pol}}\simeq\B_{\text{exp}}$ as AG. Thm.~\ref{thm:G-FunctorAlg}
thus yields
\begin{align*}
\mathcal{G}(\B_{\text{pol}},\Omega) & \simeq\mathcal{G}(\B_{\text{exp}},\Omega)\\
\Rtil(\B_{\text{pol}}) & \simeq\Rtil(\B_{\text{exp}}).
\end{align*}
This does not imply that the algebra $\mathcal{G}(\B_{\text{exp}},\R)$
is useless, because we still have the fact that the Cauchy problem
\eqref{eq:ODE} has no solution in $\mathcal{G}(\B_{\text{pol}},\R)$.
Nonetheless, we can say that the isomorphism $\lambda\in\textsc{AG}_{1}(\B_{\text{exp}},\B_{\text{pol}})$
transforms \eqref{eq:ODE} into
\begin{equation}
\begin{cases}
x'(t)-[-\log\eps]\cdot x(t)=0;\\
x(0)=1.
\end{cases}\label{eq:TransformedODE}
\end{equation}
Therefore, \eqref{eq:TransformedODE} has solution in $\mathcal{G}(\B_{\text{pol}},\R)$
if and only if \eqref{eq:ODE} has solution in $\mathcal{G}(\B_{\text{exp}},\R)$.

\noindent This example is generalized in the following theorem, where
we talk, essentially for the sake of simplicity, of ODE.
\begin{thm}
\label{thm:iso}Let $\mathbb{I\in\textsc{Ind}}$ be a set of indices
and let $b$, $c\in\R^{I}$ be infinite nets, i.e.~such that $\lim_{\eps\in\mathbb{I}}b_{\eps}=\lim_{\eps\in\mathbb{I}}c_{\eps}=+\infty$.
Set
\[
\text{\emph{AG}}(b):=\left\{ (b_{\eps}^{n})\mid n\in\N\right\} 
\]
for the AG generated by $b$ (and analogously for $c$). Assume that
$\eta$, $\lambda\in\textsc{Ind}(\mathbb{I},\mathbb{I})$ are morphisms
of $\mathbb{I}$ such that $\eta=\lambda^{-1}$:
\begin{equation}
b_{\eta(\eps)}=O_{\mathbb{\mathcal{I}}}\left(c_{\eps}\right)\ ,\ c_{\lambda(\eps)}=O_{\mathbb{\mathcal{I}}}\left(b_{\eps}\right)\text{ as }\eps\in\mathbb{I}.\label{eq:b-c-iso-relations}
\end{equation}
Then
\begin{enumerate}[leftmargin=*,label=(\roman*),align=left ]
\item \label{enu:AGb-iso-AGc}$\text{\emph{AG}}(b)\simeq\text{\emph{AG}}(c)$
as AG;
\item \label{enu:isoAlgAndCGN}$\mathcal{G}(\text{\emph{AG}}(b),\Omega)\simeq\mathcal{G}(\text{\emph{AG}}(c),\Omega)$
and $\Rtil(\text{\emph{AG}}(b))\simeq\Rtil(\text{\emph{AG}}(c))$;
\item \label{enu:ODE}let $F=[F_{\eps}]\in\mathcal{G}(\text{\emph{AG}}(b),\R^{n}\times\R)$,
$\bar{x}=[\bar{x}_{\eps}]\in\Rtil^{n}$, $\bar{t}=[\bar{t}_{\eps}]\in\Rtil$.
Then the Cauchy problem
\begin{equation}
\begin{cases}
x'(t)=F(x(t),t);\\
x(\bar{t})=\bar{x},
\end{cases}\label{eq:ODE-b}
\end{equation}
has a solution $x\in\mathcal{G}(\text{\emph{AG}}(b),(t_{1},t_{2}))$
if and only if the Cauchy problem
\begin{equation}
\begin{cases}
y'(t)=\left[F_{\lambda(\eps)}\right](y(t),t);\\
y\left(\left[\bar{t}_{\lambda(\eps)}\right]\right)=\left[\bar{x}_{\lambda(\eps)}\right],
\end{cases}\label{eq:ODE-c}
\end{equation}
has a solution $y\in\mathcal{G}(\text{\emph{AG}}(c),(t_{1},t_{2}))$.
\end{enumerate}
\end{thm}
\begin{proof}
Assumption \eqref{eq:b-c-iso-relations} yields $\left(b^{n}\circ\eta\right)(\eps)=b_{\eta(\eps)}^{n}=O_{\mathbb{\mathcal{I}}}\left(c_{\eps}^{n}\right)$
so $\R_{M}(\text{AG}(b)\circ\eta)=\R_{M}(\text{AG}(c))$. Analogously,
we have $\R_{M}(\text{AG}(c)\circ\lambda)=\R_{M}(\text{AG}(b))$.
This shows that $\eta$ and $\lambda$ are morphisms of AG, and hence
it proves \ref{enu:AGb-iso-AGc}. Property \ref{enu:isoAlgAndCGN}
follows from the functorial property of $\mathcal{G}(-,\Omega):\text{AG}_{1}\ra\textsc{Alg}_{\R}$.
To show \ref{enu:ODE}, let $x=[x_{\eps}]\in\mathcal{G}(\text{\emph{AG}}(b),(t_{1},t_{2}))$
be a solution of \eqref{eq:ODE-b} and set $y:=\mathcal{G}\left(\lambda,1_{(t_{1},t_{2})}\right)(x)=\left[x_{\lambda(\eps)}\right]$.
Therefore, Thm.~\ref{thm:G-FunctorAlg} and \eqref{eq:diffAlgMorph}
yield the conclusion.
\end{proof}
For instance, if $b$, $c:(0,1]\ra(0,1]$ are homeomorphisms such
that $\lim_{\eps\to0^{+}}b_{\eps}=0=\lim_{\eps\to0^{+}}c_{\eps}$,
then $\eta:=c\circ b^{-1}$ and $\lambda:=\eta^{-1}$ verify the assumptions
of this theorem.

\section{\label{sec:Col_n}The Category of Colombeau Algebras}

In this section, we want to show that the Colombeau AG algebra, the
related derivation of generalized functions and the embedding of distributions
are all functorial constructions with respect to the change of AG.
Although in this section we work on an arbitrary set of indices, we
restrict our study only to embeddings of Schwartz's distributions
defined through a mollifier. Therefore, we are going to deal with
mollifiers with null positive moments, namely with function $\rho\in\mathcal{S}(\mathbb{R}^{n})$
such that $\int\rho(x)x^{k}\diff{x}=0$ for every $k\in\N^{n}$, $|k|\geq1$.
We call \emph{Colombeau mollifier} any such function.
\begin{defn}
\label{def:pAG}Let $n\in\N_{>0}$ be a fixed natural number. Then
$\textsc{pAG}$ denotes the \emph{category of principal AG}, whose
objects are pairs $(b,\B)$, where $\B$ is a principal AG on a set
of indices $\mathbb{I}$ and $b\in\B$ is a generator of $\B$. Arrows
$f\in\textsc{pAG}((b_{1},\B_{1}),(b_{2},\B_{2}))$ are morphisms $f\in\textsc{AG}_{1}(\B_{1},\B_{2})$
of AG that preserve the generator, i.e.~such that $b_{1}\circ f=b_{2}$.
Let us note that if $f\in\textsc{pAG}((b_{1},\B_{1}),(b_{2},\B_{2}))$
and $g\in\textsc{pAG}((b_{2},\B_{2}),(b_{3},\B_{3}))$, then the composition
in $\textsc{pAG}$ is given as in $\textsc{AG}_{1}$, i.e.~by $f\circ g$
because $f:I_{2}\ra I_{1}$, $g:I_{3}\ra I_{2}$.
\end{defn}
\noindent Whilst the previous category acts as domain in the Colombeau
construction, the following \emph{category of Colombeau algebras}
acts as codomain.
\begin{defn}
\label{def:Col_n}Let $\textsc{DAlg}_{\R}$ be the category of differential
real algebras. We say that $(G,\partial,i)\in\textsc{Col}_{n}$ if:
\begin{enumerate}[leftmargin=*,label=(\roman*),align=left ]
\item $(G,\partial):\left(\mathcal{T}\R^{n}\right)^{\text{op}}\ra\textsc{DAlg}_{\R}$
is a functor (i.e.~it is a presheaf of differential real algebras).
In particular, $\partial_{\Omega}^{\alpha}:G(\Omega)\ra G(\Omega)$
is a derivation for all $\alpha\in\N^{n}$.
\item If we think at both functors $\mathcal{D}'$, $G:\left(\mathcal{T}\R^{n}\right)^{\text{op}}\ra\textsc{Vect}_{\R}$
with values in the category of real vector spaces, then $i:\mathcal{D}'\ra G$
is a natural transformation such that $\text{ker}\left(i_{\Omega}\right)=\{0\}$
for every $\Omega\in\mathcal{T}\R^{n}$.
\end{enumerate}

\noindent Moreover, for every $\Omega\in\mathcal{T}\R^{n}$, we have:
\begin{enumerate}[leftmargin=*,label=(\roman*),align=left,start=3]
\item $\Coo(\Omega)$ is a subalgebra of $G(\Omega)$;
\item $i_{\Omega}(f)=f$ for all $f\in\Coo(\Omega)$;
\item Let $D_{\Omega}^{\alpha}:\D'(\Omega)\ra\D'(\Omega)$ be the $\alpha\in\N^{n}$
derivation of distributions, then the following diagram commutes
\begin{equation}
\xymatrix{\D'(\Omega)\ar[r]^{i_{\Omega}}\ar[d]_{D_{\Omega}^{\alpha}} & G(\Omega)\ar[d]^{\partial_{\Omega}^{\alpha}}\\
\D'(\Omega)\ar[r]_{i_{\Omega}} & G(\Omega)
}
\label{eq:derivationsCommute}
\end{equation}

\end{enumerate}

\noindent An arrow $\varphi\in\textsc{Col}_{n}\left((G,\partial,i),(H,d,j)\right)$
is a natural transformation $\phi:(G,\partial)\ra(H,d)$ such that
for all $\Omega\in\mathcal{T}\R^{n}$ the following diagram commutes:
\begin{equation}
\xymatrix{\D'(\Omega)\ar[r]^{j_{\Omega}}\ar[d]_{i_{\Omega}} & H(\Omega)\\
G(\Omega)\ar[ur]_{\varphi_{\Omega}}
}
\label{eq:arrowsCol_n}
\end{equation}

\end{defn}
\noindent The following results prove the goal of the present section:
\begin{lem}
\label{lem:ColCategory}$\textsc{Col}_{n}$ is a category.\end{lem}
\begin{proof}
For every object $(G,\partial,i)\in\textsc{Col}_{n}$, the identity
$1_{G(\Omega)}:G(\Omega)\ra G(\Omega)$ serves as the identity arrow
of $(G,\partial,i)$ in $\textsc{Col}_{n}$. To conclude the proof,
it is sufficient to consider the composition of arrows. Let $\phi\in\textsc{Col}_{n}\left((G,\partial,i),(H,d,j)\right)$,
$\psi\in\textsc{Col}_{n}\left((H,d,j),(F,D,k)\right)$, then $\psi\circ\phi:(G,\partial)\ra(F,D)$
is a natural tranformation and the following diagram commutes:

\[
\xymatrix{G(\Omega)\ar[rd]_{\phi_{\Omega}} & \D'(\Omega)\ar[l]^{i_{\Omega}}\ar[d]_{j{}_{\Omega}}\ar[r]_{k_{\Omega}} & F(\Omega)\\
 & H(\Omega)\ar[ru]_{\psi_{\Omega}}
}
\]
In particular, we have that 
\[
\xymatrix{\D'(\Omega)\ar[r]^{k_{\Omega}}\ar[d]_{i{}_{\Omega}} & F(\Omega)\\
G(\Omega)\ar[ru]_{\left(\psi\circ\phi\right)_{\Omega}}
}
\]
commutes, namely $\psi\circ\phi\in\textsc{Col}_{n}\left((G,\partial,i),(F,D,k)\right)$.\end{proof}
\begin{thm}
\label{thm:ColFunctors}Let $\rho$ be a Colombeau mollifier. For
each $(b,\B)\in\textsc{pAG}$, set
\[
\Colomb_{n}^{\rho}(b,\B):=(\mathcal{G}(\B,-),\partial,i_{b}^{\rho}),
\]
where $i_{b}^{\rho}$ is the usual embedding defined using the generator
$b$ and the fixed mollifier $\rho$ (see \cite[Sec.~4]{GiLu15} for
details). For $f\in\textsc{pAG}((b_{1},\B_{1}),(b_{2},\B_{2}))$ and
$\Omega\in\mathcal{T}\R^{n}$, set
\begin{equation}
\Colomb_{n}^{\rho}(f)_{\Omega}:[u_{\eps_{1}}]\in\mathcal{G}(\B_{1},\Omega)\mapsto\left[u_{f(\eps_{2})}\right]\in\mathcal{G}(\B_{2},\Omega).\label{eq:Embeddings}
\end{equation}
Then
\[
\Colomb_{n}^{\rho}:\textsc{pAG}\ra\textsc{Col}_{n}
\]
is a functor.\end{thm}
\begin{proof}
The property $\Colomb_{n}^{\rho}(b,\B)\in\textsc{Col}_{n}$ for every
$(b,\B)\in\textsc{pAG}$ is a consequence of the results about Colombeau
principal AG-algebras and embeddings of distributions proved in \cite[Sec.~3 and 4]{GiLu15}.

\noindent We are left to prove the properties of $\Colomb_{n}^{\rho}$
with respect to arrows. First of all, let us prove that
\begin{align}
\Colomb_{n}^{\rho}(f) & \in\textsc{Col}_{n}(\Colomb_{n}^{\rho}(b_{1},\B_{1}),\Colomb_{n}^{\rho}(b_{2},\B_{2}))=\nonumber \\
 & =\textsc{Col}_{n}\left((\mathcal{G}(\B_{1},-),\partial,i_{b_{1}}^{\rho}),(\mathcal{G}(\B_{2},-),\partial,i_{b_{2}}^{\rho})\right)\label{eq:arrowsCol-f}
\end{align}
for every $f\in\textsc{pAG}((b_{1},\B_{1}),(b_{2},\B_{2}))$. Theorem
\ref{thm:G-FunctorAlg} gives that $\Colomb_{n}^{\rho}(f)_{\Omega}:\mathcal{G}(\B_{1},\Omega)\ra\mathcal{G}(\B_{2},\Omega)$
is a morphism of $\R$-algebras (we recall that $\mathcal{G}(\B,\Omega):=\mathcal{G}(\B,\B,\Omega)$
for every AG $\B$). From \cite[Thm.~4.7]{GiLu15}, it suffices to
prove the commutativity of the diagram 
\[
\xymatrix{\D'(\Omega)\ar[r]^{\left(i_{b_{2}}^{\rho}\right)_{\Omega}}\ar[d]_{\left(i_{b_{1}}^{\rho}\right)_{\Omega}} & \mathcal{G}(\B_{2},\Omega)\\
\mathcal{G}(\mathcal{B}_{1},\Omega)\ar[ur]_{\Colomb_{n}^{\rho}(f)_{\Omega}}
}
\]
only for compactly supported $T\in\D'(\Omega)$. In this case, we
have
\begin{align*}
\Colomb_{n}^{\rho}(f)_{\Omega}\left[\left(i_{b_{1}}^{\rho}\right)_{\Omega}(T)\right] & =\Colomb_{n}^{\rho}(f)_{\Omega}\left[T*b_{1\eps_{1}}\odot\rho\right]=\\
 & =[T*b_{1f(\eps_{2})}\odot\rho]=[T*b_{2\eps_{2}}\odot\rho]=\\
 & =\left(i_{b_{2}}^{\rho}\right)_{\Omega}(T).
\end{align*}
If $\Omega'\subseteq\Omega$, then
\begin{align*}
\Colomb_{n}^{\rho}(f)_{\Omega'}\left\{ [u_{\eps_{1}}]|_{\Omega'}\right\}  & =\Colomb_{n}^{\rho}(f)_{\Omega'}\left[u_{\eps_{1}}|_{\Omega'}\right]=\\
 & =[u_{f(\eps_{2})}|_{\Omega'}]=\left\{ \Colomb_{n}^{\rho}(f)_{\Omega}\left[u_{\eps_{1}}\right]\right\} |_{\Omega'}.
\end{align*}
This shows that $\Colomb_{n}^{\rho}(f):\mathcal{G}(\B_{1},-)\ra\mathcal{G}(\B_{2},-)$
is a natural transformation. To show \eqref{eq:arrowsCol-f}, there
remains to prove that $\Colomb_{n}^{\rho}(f)_{\Omega}$ is a morphism
of differential algebras:
\begin{align*}
\Colomb_{n}^{\rho}(f)_{\Omega}\left\{ \partial^{\alpha}[u_{\eps_{1}}]\right\}  & =\Colomb_{n}^{\rho}(f)_{\Omega}\left[\partial u_{\eps_{1}}^{\alpha}\right]=\\
 & =[\partial^{\alpha}u_{f(\eps_{2})}]=\partial^{\alpha}\left\{ \Colomb_{n}^{\rho}(f)_{\Omega}\left[u_{\eps_{1}}\right]\right\} .
\end{align*}
Since $1_{I}\in\textsc{pAG}((b,\B),(b,\B))$ is the identity in the
category $\textsc{pAG}$, it is immediate to see that $\Colomb_{n}^{\rho}(1_{I})=1_{\mathcal{G}(\B,-)}=1_{\Colomb_{n}^{\rho}(b,\B)}$
from the definition of the map \eqref{eq:Embeddings}. Finally, let
\begin{align*}
f & \in\textsc{pAG}((b_{1},\B_{1}),(b_{2},\B_{2}))\\
g & \in\textsc{pAG}((b_{2},\B_{2}),(b_{3},\B_{3})).
\end{align*}
Let $[u_{\eps_{1}}]\in\mathcal{G}(\B_{1},\Omega)$, then 
\begin{align*}
\Colomb_{n}^{\rho}(g)_{\Omega}\left\{ \Colomb_{n}^{\rho}(f)_{\Omega}([u_{\varepsilon_{1}}])\right\}  & =\Colomb_{n}^{\rho}(g)_{\Omega}([u_{f(\varepsilon_{2})}])=\\
 & =[u_{f(g(\varepsilon_{3}))}]=\Colomb_{n}^{\rho}(f\circ g)_{\Omega}([u_{\varepsilon_{1}}]).
\end{align*}

\end{proof}

\section{Conclusions}

We are forced to consider a different AG when we have to deal with
particular differential problems, whose solutions grow more than polynomially
in $\eps$. It is therefore natural to search for a notion of morphism
of AG, and to see whether Colombeau AG constructions behave in the
correct way with respect to these morphisms. The results of Sec.~\ref{sec:The-categories-Ag21},
\ref{sec:An-unexpected-isomorphism}, \ref{sec:Col_n} show that both
the construction of the differential algebra and that of the embedding
by means of a mollifier are functorial with respect to a natural notion
of morphism of AG. As shown in Sec.~\ref{sec:An-unexpected-isomorphism},
this permits to relate differential problems solved for different
AG.


\begin{thebibliography}{10}
\bibitem{Col84}Colombeau, J.F., \emph{New generalized functions and
multiplication of distributions.} North-Holland, Amsterdam, 1984.

\bibitem{Col85}Colombeau, J.F., \emph{Elementary introduction to
new generalized functions.} North-Holland, Amsterdam, 1985.

\bibitem{Col92}Colombeau, J.F., \emph{Multiplication of distributions
- A tool in mathematics, numerical engineering and theoretical Physics}.
Springer-Verlag, Berlin Heidelberg, 1992.

\bibitem{CKKR}Cutland, N., Kessler, C., Kopp, E., Ross, D. On Cauchy's
notion of infinitesimal. \emph{British J. Philos. Sci.}, 39(3):375-378,
1988.

\bibitem{GKOS}Grosser, M., Kunzinger, M., Oberguggenberger, M., Steinbauer,
R., \emph{Geometric theory of generalized functions}, Kluwer, Dordrecht,
2001.

\bibitem{MO92}Oberguggenberger, M., \emph{Multiplication of Distributions
and Applications to Partial Differential Equations}. Vol. 259. Pitman
Research Notes in Mathematics. Harlow, U.K.: Longman, 1992.

\bibitem{Del09}Delcroix, A., Topology and functoriality in $(\maC,\mathcal{E},\mathcal{P})$-algebras.
Application to singular differential problems. J. Math. Anal. Appl.
359 (2009) 394\textendash 403.

\bibitem{DPHV}Delcroix, A., Pilipovi\'{c}, S., Hasler, M.F., Valmorin,
V., Sequence spaces with exponent weights. Realizations of Colombeau
type algebras. Dissertationes Mathematicae 447, pp. 1 - 73, 2007.

\bibitem{DelSca98}Delcroix, A., Scarpalezos, D., Asymptotic scales-asymptotic
algebras. Integral Transforms and Special Functions, 1998, Vol. 6,
No. 1-4, pp. 181-190.

\bibitem{DelSca00}Delcroix, A., Scarpalezos, D., Topology on Asymptotic
Algebras of Generalized Functions and Applications. Monatshefte für
Mathematik 129, pp. 1-14, 2000.

\bibitem{GiLu15}Giordano P., Luperi Baglini L., Asymptotic gauges:
generalization of Colombeau type algebras. Accepted for publication
on Mathematische Nachrichten. See arXiv:1408.1585.

\bibitem{GiNi14}Giordano P., Nigsch E., Unifying order structures
for Colombeau algebras. Mathematische Nachrichten, 1\textendash 17
(2015). DOI 10.1002/mana.201400277.

\bibitem{GiWu14}Giordano P., Wu E., \emph{Categorical frameworks
for generalized functions}. To appear in Arabian Journal of Mathematics,
2014.

\bibitem{Has11}Hasler, M.F., On the relation between $(\maC,\mathcal{E},\mathcal{P})$-algebras
and asymptotic algebras. Document de travail 2011-07. Décembre 2011.
Centre d'Etude et de Recherche en Economie, Gestion, Modélisation
et Informatique Appliquée.

\bibitem{Lig97}Ligeza, J., Generalized periodic solutions of ordinary
differential equations in the Colombeau algebra. Annales Mathematicae
Silesianae 11. Katowice, pp. 67 - 87, 1997.

\bibitem{Lig98}Ligeza, J., Remarks on generalized solutions of ordinary
differential equations in the Colombeau algebra. Mathematica Bohemica
123, No. 3, pp. 301-316, 1998.

\bibitem{SteVic09}Steinbauer, R., Vickers, J.A.\emph{, On the Geroch\textendash Traschen
class of metrics}, Class. Quantum Grav. \textbf{26}, 2009.

\bibitem{ToVe08} Todorov, T.D., Vernaeve, H., Full algebra of generalized
functions and non-standard asymptotic analysis. \emph{Log. Anal.}
\textbf{1} (2008), 205-234.\end{thebibliography}
\end{document}